\documentclass[a4paper,11pt]{amsart}
\usepackage{amsmath,amssymb,amsfonts,amsthm}

\usepackage[alphabetic]{amsrefs }
\usepackage{bm}

\usepackage{graphicx}
\usepackage{ascmac}
\usepackage[all]{xy}
\usepackage{amsthm}
\usepackage[top=30mm,bottom=30mm,left=30mm,right=30mm]{geometry}
\usepackage{tikz}
\usetikzlibrary{patterns}
\usepackage{pxpgfmark}
\usepackage{multirow}
\usetikzlibrary{intersections,arrows,calc,decorations.markings}
\newtheorem{theorem}{Theorem}[section]

\newtheorem{proposition}[theorem]{Proposition}

\newtheorem{corollary}[theorem]{Corollary}
\newtheorem{lemma}[theorem]{Lemma}

\theoremstyle{definition}
\newtheorem{remark}[theorem]{Remark}
\newtheorem{example}[theorem]{Example}
\newtheorem{definition}[theorem]{Definition}

\makeatletter
 
 \@addtoreset{equation}{section}
\makeatother

\def\vv{\mathbf{v}}

\def\TT{\mathbb{T}}
\def\RR{\mathbb{R}}

\def\ZZ{\mathbb{Z}}
\def\QQ{\mathbb{Q}}

\title[Cluster duality between Calkin-Wilf tree and Stern-Brocot tree]{Cluster duality between \\Calkin-Wilf tree and Stern-Brocot tree}
\author{Yasuaki Gyoda}
\address{Graduate School of Mathematical Sciences, The University of Tokyo, 3-8-1 Komaba Meguro-ku Tokyo 153-8914, Japan}
\email{gyoda-yasuaki@g.ecc.u-tokyo.ac.jp}
\begin{document}
\begin{abstract}
We find a duality between two well-known trees, the Calkin-Wilf tree and the Stern-Brocot tree, derived from cluster algebra theory. The vertex sets of these trees are the set of positive rational numbers, and they have cluster structures induced by a one-punctured torus. In particular, the Calkin-Wilf tree is an example of the structure given by initial-seed mutations.
\end{abstract}
\maketitle
\tableofcontents
\section{Introduction}
In this paper, we introduce two cluster structures into \emph{the Calkin-Wilf tree} and \emph{the Stern-Brocot tree}, which are dual to each other.

The Calkin-Wilf tree was introduced by Neil Calkin and Herbert S. Wilf \cite{cal-wil} to efficiently count all positive rational numbers. This is a full binary tree with positive fractions as vertices, given in the following way. The root is $\dfrac{1}{1}$ and the generating rule is that the parent $\dfrac{x}{y}$ has the following two children:
\begin{align*}
\begin{xy}(0,0)*+{\dfrac{x}{y}}="1",(-15,-15)*+{\dfrac{x}{x+y}}="2",(15,-15)*+{\dfrac{x+y}{y}}="3", \ar@{-}"1";"2"\ar@{-}"1";"3"
\end{xy}.
\end{align*}
Here, we do not reduce fractions even if they are reducible. We label edges 1,2,or 3 as follows: first, we label the edge between $\dfrac{1}{1}$ and $\dfrac{2}{1}$ with 1, and the edge between $\dfrac{1}{1}$ and $\dfrac{1}{2}$ with 2. Moreover, for
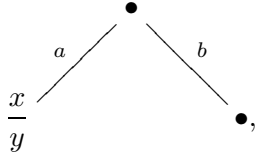
\begin{align*}
\begin{xy}(0,0)*+{\bullet}="1",(-15,-15)*+{\dfrac{x}{y}}="2",(15,-15)*+{\bullet,}="3", \ar@{-}_a"1";"2"\ar@{-}^b"1";"3"
\end{xy}
\end{align*}
we label edges between $\dfrac{x}{y}$ and their children as follows:
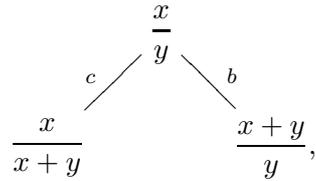
\begin{align*}
\begin{xy}(0,0)*+{\dfrac{x}{y}}="1",(-15,-15)*+{\dfrac{x}{x+y}}="2",(15,-15)*+{\dfrac{x+y}{y},}="3", \ar@{-}_c"1";"2"\ar@{-}^b"1";"3"
\end{xy}
\end{align*}
where $\{a,b,c\}=\{1,2,3\}$.
For
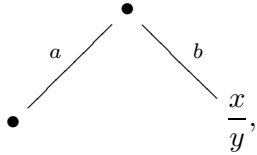
\begin{align*}
\begin{xy}(0,0)*+{\bullet}="1",(-15,-15)*+{\bullet}="2",(15,-15)*+{\dfrac{x}{y},}="3", \ar@{-}_a"1";"2"\ar@{-}^b"1";"3"
\end{xy}
\end{align*}
we label edges between $\dfrac{x}{y}$ and their children as follows:

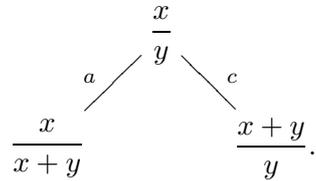
\begin{align*}
\begin{xy}(0,0)*+{\dfrac{x}{y}}="1",(-15,-15)*+{\dfrac{x}{x+y}}="2",(15,-15)*+{\dfrac{x+y}{y}.}="3", \ar@{-}_a"1";"2"\ar@{-}^c"1";"3"
\end{xy}
\end{align*}

The first few terms are as follows:
\begin{align*}
\begin{xy}(40,0)*+{1/1}="2",(55,16)*+{2/1}="4",(55,-16)*+{1/2}="5", 
(80,24)*+{3/1}="6",(80,8)*+{2/3}="7",(80,-8)*+{3/2}="8",(110,28)*+{4/1\cdots}="10",(110,20)*+{3/4\cdots}="11",(110,12)*+{5/3\cdots}="12",(110,4)*+{2/5\cdots}="13",(110,-4)*+{5/2\cdots}="14",(110,-12)*+{3/5\cdots}="15",(110,-20)*+{4/3\cdots}="16",(110,-28)*+{1/4\cdots}="17",
(80,-24)*+{1/3}="9", \ar@{-}^1"2";"4"\ar@{-}_2"2";"5"\ar@{-}^3"4";"6"\ar@{-}_2"4";"7"\ar@{-}^1"5";"8"\ar@{-}"5";"9"_3\ar@{-}^1"6";"10"\ar@{-}_2"6";"11"\ar@{-}^3"7";"12"\ar@{-}_1"7";"13"\ar@{-}^2"8";"14"\ar@{-}_3"8";"15"\ar@{-}^1"9";"16"\ar@{-}_2"9";"17"
\end{xy}.
\end{align*}

It is easy to verify that all positive rational numbers appear exactly once in the tree and they are all irreducible fractions.

On the other hand, the Stern-Brocot tree\footnote{It is also called the \emph{Farey tree}.} is named after Moritz Stern, Achille Brocot, and their researches in 1800's \cites{stern,brocot}. This is a full binary tree given in the following way: we put $\dfrac{0}{1}$ on the bottom and $\dfrac{1}{0}$ on the top. First, place $\dfrac{1}{1}$ between $\dfrac{0}{1}$ and $\dfrac{1}{0}$. Then inductively place $\dfrac{a+c}{b+d}$ between $\dfrac{a}{b}$ and $\dfrac{c}{d}$. The children of $\dfrac{a}{b}$ are the left and right neighbors of $\dfrac{a}{b}$. The first few terms are as follows:
 
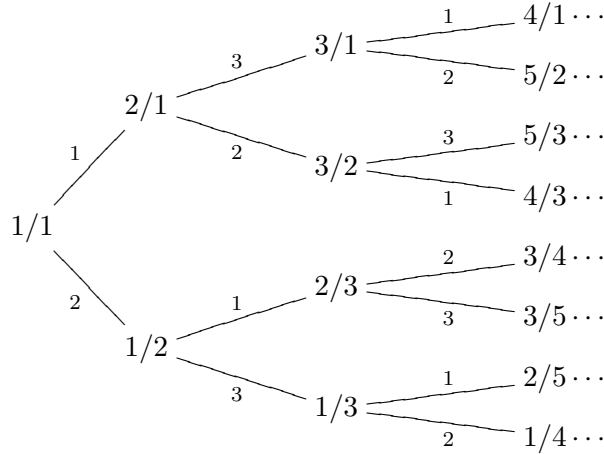
\vspace{2mm}
\begin{align*}
\begin{xy}(40,0)*+{1/1}="2",(55,16)*+{2/1}="4",(55,-16)*+{1/2}="5", 
(80,24)*+{3/1}="6",(80,8)*+{3/2}="7",(80,-8)*+{2/3}="8",(110,28)*+{4/1\cdots}="10",(110,20)*+{5/2\cdots}="11",(110,12)*+{5/3\cdots}="12",(110,4)*+{4/3\cdots}="13",(110,-4)*+{3/4\cdots}="14",(110,-12)*+{3/5\cdots}="15",(110,-20)*+{2/5\cdots}="16",(110,-28)*+{1/4\cdots}="17",
(80,-24)*+{1/3}="9", \ar@{-}^1"2";"4"\ar@{-}_2"2";"5"\ar@{-}^3"4";"6"\ar@{-}_2"4";"7"\ar@{-}^1"5";"8"\ar@{-}_3"5";"9"\ar@{-}^1"6";"10"\ar@{-}_2"6";"11"\ar@{-}^3"7";"12"\ar@{-}_1"7";"13"\ar@{-}^2"8";"14"\ar@{-}_3"8";"15"\ar@{-}^1"9";"16"\ar@{-}_2"9";"17"
\end{xy}.
\end{align*}
Here, we do not reduce fractions even if they are reducible. The labeling rule is the same as the Calkin-Wilf tree's rule. 
As with the Calkin-Wilf tree, it is easy to verify that all positive rational numbers appear exactly once in the tree, and they are all irreducible fractions.
\begin{remark}\label{stern-brocot-remark}
The Stern-Brocot tree is also constructed in the following way:
first, we consider the \emph{Farey triple tree}. This is a full binary tree given in the following way: the root is $\left(\dfrac{0}{1},\dfrac{1}{0},\dfrac{1}{1}\right)$, and the generation rule is that a parent $\left(\dfrac{a}{b},\dfrac{c}{d},\dfrac{e}{f}\right)$ has the following two children: if the second largest fraction is (i) $\dfrac{a}{b}$, (ii) $\dfrac{c}{d}$, (iii)$\dfrac{d}{e}$, then 
\begin{align*}
\begin{xy}(-20,10)*+{\mathrm{(i)}}="0",(0,0)*+{\left(\dfrac{a}{b},\dfrac{c}{d},\dfrac{e}{f}\right)}="1",(-12.5,-15)*+{\left(\dfrac{a}{b},\dfrac{a+e}{b+f},\dfrac{e}{f}\right)}="2",(12.5,-15)*+{\left(\dfrac{a}{b},\dfrac{c}{d},\dfrac{a+c}{b+d}\right)}="3", \ar@{-}_2"1";"2"\ar@{-}^3"1";"3"
\end{xy}
\begin{xy}(-20,10)*+{\mathrm{(ii)}},(0,0)*+{\left(\dfrac{a}{b},\dfrac{c}{d},\dfrac{e}{f}\right)}="1",(-12.5,-15)*+{\left(\dfrac{c+e}{d+f},\dfrac{c}{d},\dfrac{e}{f}\right)}="2",(12.5,-15)*+{\left(\dfrac{a}{b},\dfrac{c}{d},\dfrac{a+c}{b+d}\right)}="3", \ar@{-}_1"1";"2"\ar@{-}^3"1";"3"
\end{xy}
\begin{xy}(-20,10)*+{\mathrm{(iii)}},(0,0)*+{\left(\dfrac{a}{b},\dfrac{c}{d},\dfrac{e}{f}\right)}="1",(-12.5,-15)*+{\left(\dfrac{c+e}{d+f},\dfrac{c}{d},\dfrac{e}{f}\right)}="2",(12.5,-15)*+{\left(\dfrac{a}{b},\dfrac{a+e}{b+f},\dfrac{e}{f}\right)}="3",(25,-18)*+{.}, \ar@{-}_1"1";"2"\ar@{-}^2"1";"3"
\end{xy}
\end{align*}
The Stern-Brocot tree is a full binary tree obtained from the Farey triple tree by replacing each vertex with the second largest fraction of it.
\end{remark}

It has been pointed out that there are several relations between these two trees. For example, Backhouse and Ferreira found a relation between these two trees and \emph{matrix trees} \cites{backfer,backfer2,stan}. In this paper, we introduce a new relation between these trees derived from cluster algebra theory. 

Let us consider a one-punctured torus and its triangulations (see Figure \ref{fig:triangulationtorus}).
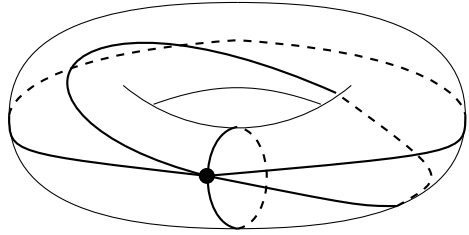
\begin{figure}[ht]
    \centering
    \caption{Triangulation of one-punctured torus}
    \label{fig:triangulationtorus}
    \vspace{2mm}
\begin{tikzpicture}[baseline=0mm]
\coordinate (u) at (0,1.5);
\coordinate (d) at (0,-1.5);
\coordinate (m1) at (-3,0);
\coordinate (m2) at (-1.5,0.4);
\coordinate (m3) at (1.5,0.4);
\coordinate (m4) at (3,0);
\coordinate(m5)at (-1.1,0.15);
\coordinate(m6)at (1.1,0.15);
\coordinate (p) at (-0.4,-0.8);
\coordinate (m7) at (0,-0.15);
 \fill(p) circle(3pt);
 \draw (u) to [out=180,in=90] (m1);
 \draw (m1) to [out=-90,in=180] (d);
 \draw (d) to [out=0,in=-90] (m4);
 \draw (m4) to [out=90,in=0] (u);
 \draw (m2) to [out=-40,in=-140] (m3); 
 \draw (m5) to [out=20,in=160] (m6); 
 \draw[thick] (d) to [out=170,in=-170] (m7);
 \draw[thick,dashed] (d) to [out=10,in=-10] (m7);
 \draw[thick] (p).. controls (-3,-0.5)and (-3, -0.5) .. (-3,0);
 \draw[thick] (p).. controls (3,-0.5)and (3, -0.5) .. (3,0);
 \draw[thick,dashed] (-3,0) .. controls (-3,0)and (-3,0.8) .. (0,1) to (0,1).. controls (3,0.8)and (3,0) .. (3,0);
\draw [thick](p) .. controls (1.6,-1.2)and (1.6, -1.2) ..  (2.1,-1.2); 
\draw [thick,dashed](2.1,-1.2)  .. controls (2.8,-0.8)and (2.8, -0.8) .. (1.3,0.3); 
\draw [thick](1.3,0.3)  .. controls (-2.5,2)and (-3.5, 0) .. (p); 
\end{tikzpicture}
\end{figure}
We fix a triangulation $L=(\ell_1,\ell_2,\ell_3)$. For an arc $\ell$ included in some triangulation, we set $D(L,\ell)=\begin{bmatrix}d_1\\d_2\\d_3\end{bmatrix}$, where $d_i$ is the intersection number of $\ell$ and $\ell_i$ (note that the intersection number of itself is $-1$). It is called the \emph{intersection vector}.
Next, we consider an operation called a \emph{flip}, which we will define strictly in Section 2. This is an operation to obtain a new triangulation by exchanging one of the edges constituting the triangulation. The Figure \ref{flip-triangulation} shows an example.
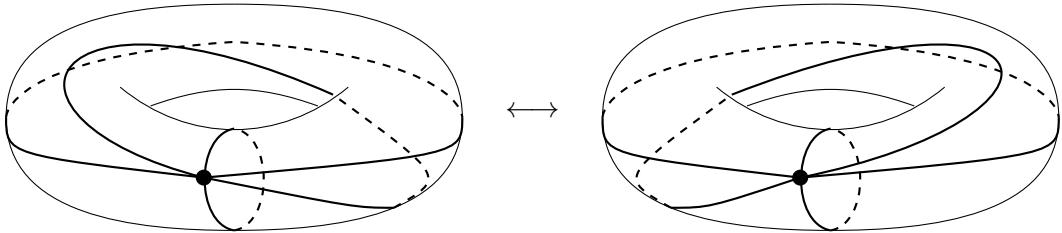
\begin{figure}[ht]
    \centering
    \caption{Flip of triangulation}
    \label{flip-triangulation}
    \vspace{2mm}
\begin{tikzpicture}[baseline=0mm]
\coordinate (u) at (0,1.5);
\coordinate (d) at (0,-1.5);
\coordinate (m1) at (-3,0);
\coordinate (m2) at (-1.5,0.4);
\coordinate (m3) at (1.5,0.4);
\coordinate (m4) at (3,0);
\coordinate(m5)at (-1.1,0.15);
\coordinate(m6)at (1.1,0.15);
\coordinate (p) at (-0.4,-0.8);
\coordinate (m7) at (0,-0.15);
 \fill(p) circle(3pt);
 \draw (u) to [out=180,in=90] (m1);
 \draw (m1) to [out=-90,in=180] (d);
 \draw (d) to [out=0,in=-90] (m4);
 \draw (m4) to [out=90,in=0] (u);
 \draw (m2) to [out=-40,in=-140] (m3); 
 \draw (m5) to [out=20,in=160] (m6); 
 \draw[thick] (d) to [out=170,in=-170] (m7);
 \draw[thick,dashed] (d) to [out=10,in=-10] (m7);
 \draw[thick] (p).. controls (-3,-0.5)and (-3, -0.5) .. (-3,0);
 \draw[thick] (p).. controls (3,-0.5)and (3, -0.5) .. (3,0);
 \draw[thick,dashed] (-3,0) .. controls (-3,0)and (-3,0.8) .. (0,1) to (0,1).. controls (3,0.8)and (3,0) .. (3,0);
\draw [thick](p) .. controls (1.6,-1.2)and (1.6, -1.2) ..  (2.1,-1.2); 
\draw [thick,dashed](2.1,-1.2)  .. controls (2.8,-0.8)and (2.8, -0.8) .. (1.3,0.3); 
\draw [thick](1.3,0.3)  .. controls (-2.5,2)and (-3.5, 0) .. (p); 
\end{tikzpicture}
\hspace{3mm}
$\longleftrightarrow$
\hspace{3mm}
\begin{tikzpicture}[baseline=0mm]
\coordinate (u) at (0,1.5);
\coordinate (d) at (0,-1.5);
\coordinate (m1) at (-3,0);
\coordinate (m2) at (-1.5,0.4);
\coordinate (m3) at (1.5,0.4);
\coordinate (m4) at (3,0);
\coordinate(m5)at (-1.1,0.15);
\coordinate(m6)at (1.1,0.15);
\coordinate (p) at (-0.4,-0.8);
\coordinate (m7) at (0,-0.15);
 \fill(p) circle(3pt);
 \draw (u) to [out=180,in=90] (m1);
 \draw (m1) to [out=-90,in=180] (d);
 \draw (d) to [out=0,in=-90] (m4);
 \draw (m4) to [out=90,in=0] (u);
 \draw (m2) to [out=-40,in=-140] (m3); 
 \draw (m5) to [out=20,in=160] (m6); 
 \draw[thick] (d) to [out=170,in=-170] (m7);
 \draw[thick,dashed] (d) to [out=10,in=-10] (m7);
 \draw[thick] (p).. controls (-3,-0.5)and (-3, -0.5) .. (-3,0);
 \draw[thick] (p).. controls (3,-0.5)and (3, -0.5) .. (3,0);
 \draw[thick,dashed] (-3,0) .. controls (-3,0)and (-3,0.8) .. (0,1) to (0,1).. controls (3,0.8)and (3,0) .. (3,0);
\draw [thick](p) .. controls (-1.6,-1.2)and (-1.6, -1.2) ..  (-2.1,-1.2); 
\draw [thick,dashed](-2.1,-1.2)  .. controls (-2.8,-0.8)and (-2.8, -0.8) .. (-1.3,0.3); 
\draw [thick](-1.3,0.3)  .. controls (3,2)and (3.5, 0) .. (p); 
\end{tikzpicture}
\end{figure}
\\
There are three ways to flip for each triangle.
Repeatedly applying flips to triangulations results in a full binary tree, called the \emph{intersection vector tree}.
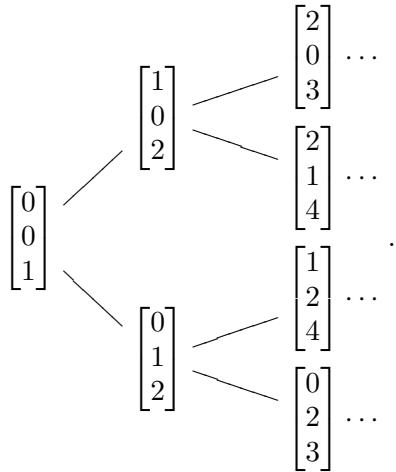
\begin{align*}
\begin{xy}(40,0)*+{\begin{bmatrix}0\\0\\1\end{bmatrix}}="2",(57,16)*+{\begin{bmatrix}1\\0\\2\end{bmatrix}}="4",(57,-16)*+{\begin{bmatrix}0\\1\\2\end{bmatrix}}="5", 
(80,24)*+{\begin{bmatrix}2\\0\\3\end{bmatrix}\cdots}="6",(80,8)*+{\begin{bmatrix}2\\1\\4\end{bmatrix}\cdots}="7",(80,-8)*+{\begin{bmatrix}1\\2\\4\end{bmatrix}\cdots}="8",(80,-24)*+{\begin{bmatrix}0\\2\\3\end{bmatrix}\cdots}="9", \ar@{-}"2";"4"\ar@{-}"2";"5"\ar@{-}"4";"6"\ar@{-}"4";"7"\ar@{-}"5";"8"\ar@{-}"5";"9"
\end{xy}.
\end{align*}
The first main theorem relates the intersection vector tree to the Stern-Brocot tree:
\begin{theorem}[Theorem \ref{thm:main1}]\label{thm:main1intro}
We consider a map
\begin{align*}
    g\colon\ZZ_{\geq0}^3 \to \QQ,\quad \begin{bmatrix}d_1\\d_2\\d_3\end{bmatrix}\mapsto \dfrac{d_1+1}{d_2+1}.
\end{align*}
The Stern-Brocot tree is obtained by replacing each vertex $v$ of the intersection vector tree with $g(v)$.
\end{theorem}
Next, we introduce a counterpart of the Calkin-Wilf tree. In contrast to the previous, we fix an arc $\ell$ and consider changing triangulations from $L$ to $L'$ by a flip. In parallel with change of triangulations, we obtain another intersection vector $D(L',\ell)$. By doing it repeatedly, we define another tree, the \emph{initial intersection vector tree}: 
\begin{align}\label{treeintro}
\begin{xy}(40,0)*+{\begin{bmatrix}0\\0\\1\end{bmatrix}}="2",(57,16)*+{\begin{bmatrix}2\\0\\1\end{bmatrix}}="4",(57,-16)*+{\begin{bmatrix}0\\2\\1\end{bmatrix}}="5", 
(80,24)*+{\begin{bmatrix}2\\0\\3\end{bmatrix}\dots}="6",(80,8)*+{\begin{bmatrix}2\\4\\1\end{bmatrix}\dots}="7",(80,-8)*+{\begin{bmatrix}4\\2\\1\end{bmatrix}\dots}="8",(80,-24)*+{\begin{bmatrix}0\\2\\3\end{bmatrix}\dots}="9", \ar@{-}^1"2";"4"\ar@{-}_2"2";"5"\ar@{-}^3"4";"6"\ar@{-}_2"4";"7"\ar@{-}^1"5";"8"\ar@{-}_3"5";"9"
\end{xy},
\end{align}
where the numbers around each edge indicate the position of the edge that will be exchanged in each triangulation. Let $d(t,t')$ be the distance from $t$ and $t'$ on the tree. The second main theorem relates the initial intersection vector tree to the Calkin-Wilf tree:
\begin{theorem}[Theorem \ref{thm:main2}]\label{thm:main2intro}
  We define a map $h$ from vertices of the initial intersection vector tree to $\QQ$
 inductively as follows: we assign the leftmost vertex $\begin{bmatrix}d_{1;t_1}\\d_{2;t_1}\\d_{3;t_1}\end{bmatrix}=\begin{bmatrix}0\\0\\1\end{bmatrix}$ to 
 $\dfrac{d_{1;t_0}+1}{d_{2;t_0}+1}=\dfrac{1}{1}$. Let $\{a,b,c\}=\{1,2,3\}$. When $D(L_t,\ell)=\begin{bmatrix}d_{1;t}\\d_{2;t}\\d_{3;t}\end{bmatrix}\mapsto \dfrac{d_{a;t}+1}{d_{b;t}+1}$, and $\begin{xy}(0,0)*+{D(L_t,\ell)}="A",(23,0)*+{D(L_{t'},\ell)}="B",\ar@{-}^k"A";"B" \end{xy}$ with $d(t_1,t)<d(t_1,t')$, 
\begin{itemize}
    \item if $k=a$, then we assign $D(L_{t'},\ell)\mapsto \dfrac{d_{c;t}+1}{d_{b;t}+1}$,
    \vspace{2mm}
    \item if $k=b$, then we assign $D(L_{t'},\ell)\mapsto \dfrac{d_{a;t}+1}{d_{c;t}+1}$.
    \end{itemize}
The Calkin-Wilf tree is obtained by replacing each vertex $v$ of the initial intersection vector tree with $h(v)$.
\end{theorem}

In the context of cluster algebra theory, we can regard the relation between Theorem \ref{thm:main1intro} and Theorem \ref{thm:main2intro} as a specialization of the \emph{$D$-matrix duality} introduced by \cite{rs}\footnote{In this paper, for the convenience of the proof, we describe the duality of the $D$-matrix, which is true only for some classes of cluster algebras. This duality property can be attributed to the duality of the $F$-matrix, which is similar to the $D$-matrix. See \cite{fg} for details.}. 

In the last section, we mention the duality of the Cohn tree and the Christoffel tree using the duality of Theorem \ref{thm:main1intro} and Theorem \ref{thm:main2intro}.

Trees in this paper and their relations are shown in Figure \ref{relationtree}. Here, each dashed arrow indicates that the tree at the end of the arrow can be constructed by extracting some of the entries contained in each vertex of the tree at the start of the arrow.

\begin{figure}[ht]
    \caption{Relation of trees}
    \label{relationtree}
    \vspace{4mm}
\scalebox{0.88}{{\begin{tikzpicture}
\node(0) at (-3,0){Farey triple}; 
\node(0') at (-3,-0.3){tree}; 
\node(1) at (0,0){Stern-Brocot tree}; 
\node(2) at (8,0){Calkin-Wilf tree}; 
\node(3) at (0,2.6){Intersection vector tree}; 
\node(3') at (0,2.2){Tree($D$)}; 
\node(4) at (8,2.6){Initial intersection};
\node(4') at (8,2.2){vector tree Tree($D^\dag$)};
\node(7) at (8,-2.4){Cohn tree}; 
\node(7') at (8,-2.8){Tree(Co)}; 
\node(9) at (0,-2.4){Christoffel tree};
\node(9') at (0,-2.8){Tree(Ch)}; 
\node(11) at (-3.5,-2.4){Christoffel triple};
\node(11') at (-3.5,-2.8){tree Tree(3Ch)};
\node(10) at (11,-2.4){Cohn triple};
\node(10') at (11,-2.8){tree Tree(3Co)}; 
\draw[<->](1) to (3');
\draw[->](9) to (1);
\draw[->](9) to (2);
\draw[->](7) to (1);
\draw[->](7) to (2);
\draw[dashed,->](10) to (7);
\draw[dashed,->](11) to (9);
\draw[dashed,<-](1) to (0);
\draw[<->](2) to (4');
\draw[<->](3) to node[pos=0.5,auto=left] {\footnotesize Duality (Proposition \ref{prop:duality})}(4);
\fill [white] (1.6,-1.5) circle [x radius=1cm, y radius=5mm];
\fill [white] (6.4,-1.5) circle [x radius=1cm, y radius=5mm];
\node(0) at (-1.5,1.3){\footnotesize Theorem \ref{thm:main1intro}  (\ref{thm:main1})/};
\node(0) at (-1,0.9){\footnotesize Remark \ref{contrarymain1}};
\node(0) at (9.6,1.3){\footnotesize Theorem \ref{thm:main2intro} (\ref{thm:main2})/};
\node(0) at (9.3,0.9){\footnotesize Remark \ref{contrarymain2}};
\node(0) at (-1.25,-1.3){\footnotesize Theorem \ref{Christoffel-main} (3)};
\node(0) at (9.3,-1.3){\footnotesize Theorem \ref{dualChristoffel-main} (3)};
\node(0) at (2,-1.8){\footnotesize Theorem \ref{Christoffel-main} (2)};
\node(0) at (6,-1.8){\footnotesize Theorem \ref{dualChristoffel-main} (2)};
\end{tikzpicture}}}
\end{figure}

\subsection*{Organization}
In Section 2, we introduce the cluster structure of one-punctured torus and define the intersection vector and matrix. In Section 3, we prove Theorem \ref{thm:main1intro} and give an explicit description of the intersection matrix. In Section 4, we prove Theorem \ref{thm:main2intro}. In Section 5, we define the Cohn tree and the Christoffel tree and describe their duality.

\subsection*{Acknowledgment}
The author appreciates the feedback offered by his supervisor Tomoki Nakanishi. The author also would like to thank Akira Kondo, Satoshi Sugiyama, and Yudai Suzuki for their insightful comments. This work was supported by JSPS KAKENHI Grant number JP20J12675.
\section{Cluster pattern from one-punctured torus}
In this section, we introduce the cluster pattern from a one-punctured torus. This is a special case of the cluster structure from marked surfaces introduced by \cite{fst}. Let $S$ be a one-punctured torus and let $p$ be the only puncture of $S$.  Consider the triangulation of $S$ by arcs with $p$ at both ends; on the universal covering of $S$, the triangulation of $S$ is given by Figure \ref{fig:triangulation}.
\begin{figure}[ht]
\caption{Triangulation of one-punctured torus (universal covering)\label{fig:triangulation}}
\vspace{3mm}
\reflectbox{\begin{tikzpicture}
\node[circle,fill=black,inner sep=1.5](1) at (0,0){}; 
\node[circle,fill=black,inner sep=1.5](2) at (0,-1){}; 
\node[circle,fill=black,inner sep=1.5](3) at (0,-2){}; 
\node[circle,fill=black,inner sep=1.5](5) at (1,0){}; 
\node[circle,fill=black,inner sep=1.5](6) at (1,-1){}; 
\node[circle,fill=black,inner sep=1.5](7) at (1,-2){}; 
\node[circle,fill=black,inner sep=1.5](8) at (2,-1){}; 
\node[circle,fill=black,inner sep=1.5](9) at (2,0){};
\node[circle,fill=black,inner sep=1.5](10) at (2,1){};  
\node[circle,fill=black,inner sep=1.5](11) at (1,1){}; 
\node[circle,fill=black,inner sep=1.5](12) at (0,1){}; 
\node[circle,fill=black,inner sep=1.5](13) at (2,-2){}; 
\node[circle,fill=black,inner sep=1.5](14) at (3,1){}; 
\node[circle,fill=black,inner sep=1.5](15) at (3,0){}; 
\node[circle,fill=black,inner sep=1.5](16) at (3,-1){}; 
\node[circle,fill=black,inner sep=1.5](17) at (3,-2){}; 
\draw(-0.5,1) to (3.5,1);
\draw(-0.5,0) to (3.5,0);
\draw(-0.5,-1) to (3.5,-1);
\draw(-0.5,-2) to (3.5,-2);
\draw(0,1.5) to (0,-2.5);
\draw(1,1.5) to (1,-2.5);
\draw(2,1.5) to (2,-2.5);
\draw(3,1.5) to (3,-2.5);
\draw(3.5,1.5) to (-0.5,-2.5);
\draw(3.5,0.5) to (0.5,-2.5);
\draw(3.5,-1.5) to (2.5,-2.5);
\draw(3.5,-0.5) to (1.5,-2.5);
\draw(2.5,1.5) to (-0.5,-1.5);
\draw(1.5,1.5) to (-0.5,-0.5);
\draw(0.5,1.5) to (-0.5,0.5);
\end{tikzpicture}}
\end{figure}
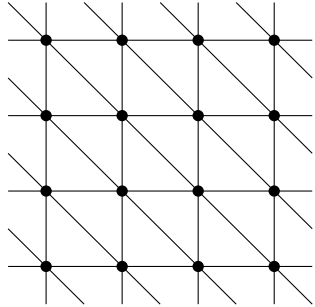
A triangulation of $S$ consists of three arcs. Hereinafter, a set of arcs constructing a triangulation is referred to as a \emph{triangulation} simply. In this paper, we define a triangulation as an ordered set. Let $L=(\ell_1,\ell_2,\ell_3)$ be a triangulation. From symmetry, we can assume without loss of generality that $\ell_1$ is a horizontal line, $\ell_2$ is a vertical line, and $\ell_3$ is a diagonal line in Figure \ref{fig:triangulation}. For $k\in\{1,2,3\}$, we define the \emph{flip} $\varphi_k(L)$ of $L$ in direction $k$ as the operation of exchanging $\ell_k$ from $L$ with another arc to obtain another triangle. Figure \ref{fig:flip} shows the triangulations of $S$ flipped from $L$ in directions 1,2, and 3, respectively.
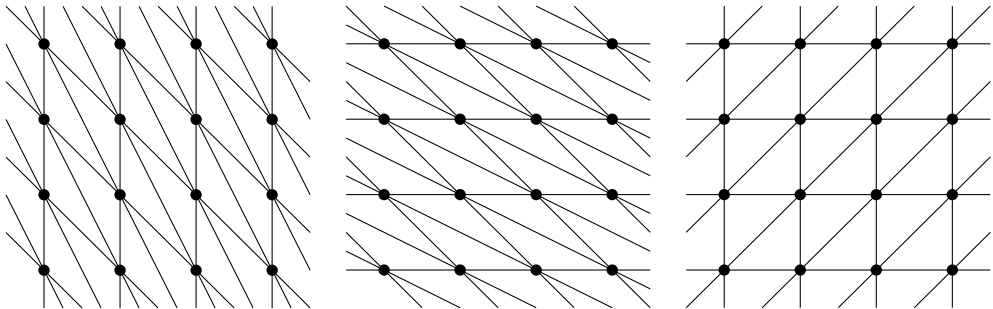
\begin{figure}[ht]
\caption{Flipped triangulation of one-punctured torus\label{fig:flip}}
\vspace{3mm}
\rotatebox{90}{\begin{tikzpicture}
\node[circle,fill=black,inner sep=1.5](1) at (0,0){}; 
\node[circle,fill=black,inner sep=1.5](2) at (0,-1){}; 
\node[circle,fill=black,inner sep=1.5](3) at (0,-2){}; 
\node[circle,fill=black,inner sep=1.5](5) at (1,0){}; 
\node[circle,fill=black,inner sep=1.5](6) at (1,-1){}; 
\node[circle,fill=black,inner sep=1.5](7) at (1,-2){}; 
\node[circle,fill=black,inner sep=1.5](8) at (2,-1){}; 
\node[circle,fill=black,inner sep=1.5](9) at (2,0){};
\node[circle,fill=black,inner sep=1.5](10) at (2,1){};  
\node[circle,fill=black,inner sep=1.5](11) at (1,1){}; 
\node[circle,fill=black,inner sep=1.5](12) at (0,1){}; 
\node[circle,fill=black,inner sep=1.5](13) at (2,-2){}; 
\node[circle,fill=black,inner sep=1.5](14) at (3,1){}; 
\node[circle,fill=black,inner sep=1.5](15) at (3,0){}; 
\node[circle,fill=black,inner sep=1.5](16) at (3,-1){}; 
\node[circle,fill=black,inner sep=1.5](17) at (3,-2){}; 
\draw(-0.5,1) to (3.5,1);
\draw(-0.5,0) to (3.5,0);
\draw(-0.5,-1) to (3.5,-1);
\draw(-0.5,-2) to (3.5,-2);
\draw(3.5,1.5) to (-0.5,-2.5);
\draw(3.5,0.5) to (0.5,-2.5);
\draw(3.5,-1.5) to (2.5,-2.5);
\draw(3.5,-0.5) to (1.5,-2.5);
\draw(2.5,1.5) to (-0.5,-1.5);
\draw(1.5,1.5) to (-0.5,-0.5);
\draw(0.5,1.5) to (-0.5,0.5);
\draw(-0.5,0.75) to (1,1.5);
\draw(-0.5,0.25) to (2,1.5);
\draw(-0.5,-0.25) to (3,1.5);
\draw(-0.5,-0.75) to (3.5,1.25);
\draw(-0.5,-1.25) to (3.5,0.75);
\draw(-0.5,-1.75) to (3.5,0.25);
\draw(-0.5,-2.25) to (3.5,-0.25);
\draw(0,-2.5) to (3.5,-0.75);
\draw(1,-2.5) to (3.5,-1.25);
\draw(2,-2.5) to (3.5,-1.75);
\end{tikzpicture}}
\hspace{2mm}
\reflectbox{\begin{tikzpicture}
\node[circle,fill=black,inner sep=1.5](1) at (0,0){}; 
\node[circle,fill=black,inner sep=1.5](2) at (0,-1){}; 
\node[circle,fill=black,inner sep=1.5](3) at (0,-2){}; 
\node[circle,fill=black,inner sep=1.5](5) at (1,0){}; 
\node[circle,fill=black,inner sep=1.5](6) at (1,-1){}; 
\node[circle,fill=black,inner sep=1.5](7) at (1,-2){}; 
\node[circle,fill=black,inner sep=1.5](8) at (2,-1){}; 
\node[circle,fill=black,inner sep=1.5](9) at (2,0){};
\node[circle,fill=black,inner sep=1.5](10) at (2,1){};  
\node[circle,fill=black,inner sep=1.5](11) at (1,1){}; 
\node[circle,fill=black,inner sep=1.5](12) at (0,1){}; 
\node[circle,fill=black,inner sep=1.5](13) at (2,-2){}; 
\node[circle,fill=black,inner sep=1.5](14) at (3,1){}; 
\node[circle,fill=black,inner sep=1.5](15) at (3,0){}; 
\node[circle,fill=black,inner sep=1.5](16) at (3,-1){}; 
\node[circle,fill=black,inner sep=1.5](17) at (3,-2){}; 
\draw(-0.5,1) to (3.5,1);
\draw(-0.5,0) to (3.5,0);
\draw(-0.5,-1) to (3.5,-1);
\draw(-0.5,-2) to (3.5,-2);
\draw(3.5,1.5) to (-0.5,-2.5);
\draw(3.5,0.5) to (0.5,-2.5);
\draw(3.5,-1.5) to (2.5,-2.5);
\draw(3.5,-0.5) to (1.5,-2.5);
\draw(2.5,1.5) to (-0.5,-1.5);
\draw(1.5,1.5) to (-0.5,-0.5);
\draw(0.5,1.5) to (-0.5,0.5);
\draw(-0.5,0.75) to (1,1.5);
\draw(-0.5,0.25) to (2,1.5);
\draw(-0.5,-0.25) to (3,1.5);
\draw(-0.5,-0.75) to (3.5,1.25);
\draw(-0.5,-1.25) to (3.5,0.75);
\draw(-0.5,-1.75) to (3.5,0.25);
\draw(-0.5,-2.25) to (3.5,-0.25);
\draw(0,-2.5) to (3.5,-0.75);
\draw(1,-2.5) to (3.5,-1.25);
\draw(2,-2.5) to (3.5,-1.75);
\end{tikzpicture}}
\hspace{2mm}
\begin{tikzpicture}
\node[circle,fill=black,inner sep=1.5](1) at (0,0){}; 
\node[circle,fill=black,inner sep=1.5](2) at (0,-1){}; 
\node[circle,fill=black,inner sep=1.5](3) at (0,-2){}; 
\node[circle,fill=black,inner sep=1.5](5) at (1,0){}; 
\node[circle,fill=black,inner sep=1.5](6) at (1,-1){}; 
\node[circle,fill=black,inner sep=1.5](7) at (1,-2){}; 
\node[circle,fill=black,inner sep=1.5](8) at (2,-1){}; 
\node[circle,fill=black,inner sep=1.5](9) at (2,0){};
\node[circle,fill=black,inner sep=1.5](10) at (2,1){};  
\node[circle,fill=black,inner sep=1.5](11) at (1,1){}; 
\node[circle,fill=black,inner sep=1.5](12) at (0,1){}; 
\node[circle,fill=black,inner sep=1.5](13) at (2,-2){}; 
\node[circle,fill=black,inner sep=1.5](14) at (3,1){}; 
\node[circle,fill=black,inner sep=1.5](15) at (3,0){}; 
\node[circle,fill=black,inner sep=1.5](16) at (3,-1){}; 
\node[circle,fill=black,inner sep=1.5](17) at (3,-2){}; 
\draw(-0.5,1) to (3.5,1);
\draw(-0.5,0) to (3.5,0);
\draw(-0.5,-1) to (3.5,-1);
\draw(-0.5,-2) to (3.5,-2);
\draw(0,1.5) to (0,-2.5);
\draw(1,1.5) to (1,-2.5);
\draw(2,1.5) to (2,-2.5);
\draw(3,1.5) to (3,-2.5);
\draw(3.5,1.5) to (-0.5,-2.5);
\draw(3.5,0.5) to (0.5,-2.5);
\draw(3.5,-1.5) to (2.5,-2.5);
\draw(3.5,-0.5) to (1.5,-2.5);
\draw(2.5,1.5) to (-0.5,-1.5);
\draw(1.5,1.5) to (-0.5,-0.5);
\draw(0.5,1.5) to (-0.5,0.5);
\end{tikzpicture}
\end{figure}
Let $\mathbb{T}_3$ be the \emph{$3$-regular tree} whose edges are labeled by the numbers 1,2,3 such that the three edges emanating from each vertex have different labels. We use the notation $\begin{xy}(0,0)*+{t}="A",(10,0)*+{t'}="B",\ar@{-}^k"A";"B" \end{xy}$
to indicate that vertices $t,t'\in \mathbb{T}_3$ are joined by an edge labeled by $k$. We fix an arbitrary vertex $t_0\in \TT_3$, which is called the \emph{rooted vertex}, and a triangulation $L$. A \emph{cluster pattern} with the initial triangulation $L$ is an assignment of a triangulation $L_t=(\ell_{1;t},\ell_{2;t},\ell_{3;t})$ to every vertex $t\in \mathbb{T}_n$ such that $L$ is assigned $t_0$ and triangulations $L_t$ and $L_{t'}$ assigned to the endpoints of any edge
$\begin{xy}(0,0)*+{t}="A",(10,0)*+{t'}="B",\ar@{-}^k"A";"B" \end{xy}$
are obtained from each other by a flip in direction $k$. We denote this assignment by $P_L \colon t\mapsto L_t$. For a cluster pattern $P_L$, when $\ell_{i;t}$ intersects with $\ell_1,\ell_2$ and $\ell_3$ at least $d_{i1},d_{i2}$ and $d_{i3}$ times on $S\setminus\{p\}$ respectively, we define the \emph{intersection vector} $D(L,\ell_{i;t})$ associated with $\ell_{i;t}$ as 
\begin{align*}
    D(L,\ell_{i;t})=\begin{bmatrix} d_{i1}\\d_{i2}\\d_{i3}\end{bmatrix}.
\end{align*}    
We regard the intersection number of the same arc as $-1$.
Furthermore, we define the \emph{intersection matrix} $D(L,L_t)$ associated with $L_t$ as 
\begin{align*}
    D(L,L_{t})=\begin{bmatrix} d_{11}&d_{12}&d_{13}\\d_{21}&d_{22}&d_{23}\\d_{31}&d_{32}&d_{33}\end{bmatrix}.
\end{align*} 
\begin{remark}
In the context of cluster algebra theory, it is known that intersection vectors and intersection matrices correspond with $d$-vectors and $D$-matrices introduced by \cites{fziv,rs}. See \cite{fst}. 
\end{remark}
\section{Intersection vector tree and Stern-Brocot tree}
Let $t_1$ be a vertex of $\TT_3$ connected with $t_0$ by an edge labeled 3, and we consider a full subtree whose vertex set is the union of $t_0$ and all vertices that are reachable to $t_1$ without going through $t_0$:
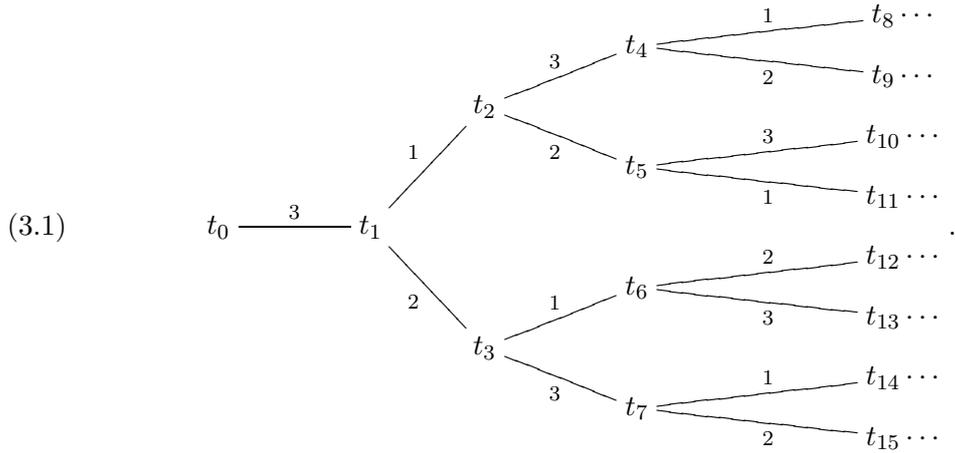
\vspace{5mm}
\begin{align}\label{subtree}
\begin{xy}(20,0)*+{t_0}="1",(40,0)*+{t_1}="2",(55,16)*+{t_2}="4",(55,-16)*+{t_3}="5", 
(75,24)*+{t_4}="6",(75,8)*+{t_5}="7",(75,-8)*+{t_6}="8",(110,28)*+{t_8\cdots}="10",(110,20)*+{t_{9}\cdots}="11",(110,12)*+{t_{10}\cdots}="12",(110,4)*+{t_{11}\cdots}="13",(110,-4)*+{t_{12}\cdots}="14",(110,-12)*+{t_{13}\cdots}="15",(110,-20)*+{t_{14}\cdots}="16",(110,-28)*+{t_{15}\cdots}="17",
(75,-24)*+{t_7}="9", \ar@{-}^{1}"2";"4"\ar@{-}_{2}"2";"5"\ar@{-}^{3}"4";"6"\ar@{-}_{2}"4";"7"\ar@{-}^{1}"5";"8"\ar@{-}_{3}"5";"9"\ar@{-}^{3}"1";"2"\ar@{-}^{1}"6";"10"\ar@{-}_{2}"6";"11"\ar@{-}^{3}"7";"12"\ar@{-}_{1}"7";"13"\ar@{-}^{2}"8";"14"\ar@{-}_{3}"8";"15"\ar@{-}^{1}"9";"16"\ar@{-}_{2}"9";"17"
\end{xy}.
\end{align}
 Let $\TT'_3$ be a full subtree of the above tree whose vertex set consists of all vertices of the above except for $t_0$.
We correspond the intersection vectors to vertices of $\TT'_3$ as 
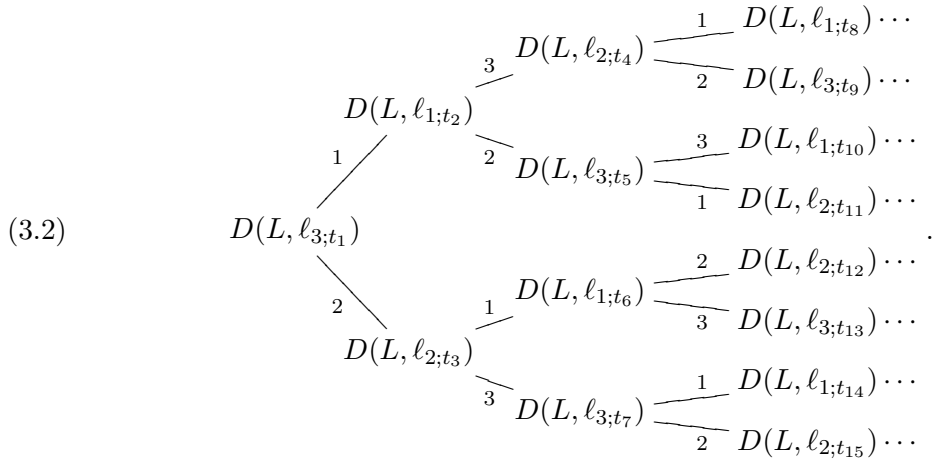
\begin{align}\label{F-tree}
\begin{xy}(40,0)*+{D(L,\ell_{3;t_1})}="2",(55,16)*+{D(L,\ell_{1;t_2})}="4",(55,-16)*+{D(L,\ell_{2;t_3})}="5", 
(77.5,24)*+{D(L,\ell_{2;t_4})}="6",(77.5,8)*+{D(L,\ell_{3;t_5})}="7",(77.5,-8)*+{D(L,\ell_{1;t_6})}="8",(110,28)*+{D(L,\ell_{1;t_8})\cdots}="10",(110,20)*+{D(L,\ell_{3;t_{9}})\cdots}="11",(110,12)*+{D(L,\ell_{1;t_{10}})\cdots}="12",(110,4)*+{D(L,\ell_{2;t_{11}})\cdots}="13",(110,-4)*+{D(L,\ell_{2;t_{12}})\cdots}="14",(110,-12)*+{D(L,\ell_{3;t_{13}})\cdots}="15",(110,-20)*+{D(L,\ell_{1;t_{14}})\cdots}="16",(110,-28)*+{D(L,\ell_{2;t_{15}})\cdots}="17",
(77.5,-24)*+{D(L,\ell_{3;t_7})}="9", \ar@{-}^{1}"2";"4"\ar@{-}_{2}"2";"5"\ar@{-}^{3}"4";"6"\ar@{-}_{2}"4";"7"\ar@{-}^{1}"5";"8"\ar@{-}_{3}"5";"9"\ar@{-}^{1}"6";"10"\ar@{-}_{2}"6";"11"\ar@{-}^{3}"7";"12"\ar@{-}_{1}"7";"13"\ar@{-}^{2}"8";"14"\ar@{-}_{3}"8";"15"\ar@{-}^{1}"9";"16"\ar@{-}_{2}"9";"17"
\end{xy}.
\end{align}
That is, if an edge labeled $k$ emanates from the left side of $t_i$ in \eqref{subtree}, we assign ${D(L,\ell_{k;t_i})}$ to $t_i$. Its first seven
vertices are as follows:
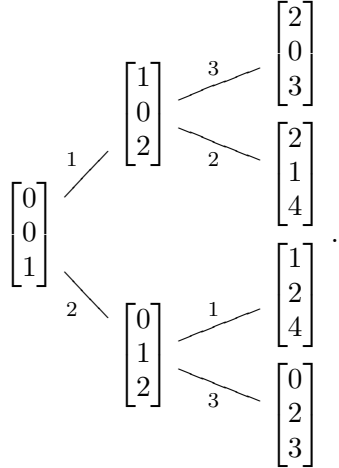
\begin{align*}
\begin{xy}(40,0)*+{\begin{bmatrix}0\\0\\1\end{bmatrix}}="2",(55,16)*+{\begin{bmatrix}1\\0\\2\end{bmatrix}}="4",(55,-16)*+{\begin{bmatrix}0\\1\\2\end{bmatrix}}="5", 
(75,24)*+{\begin{bmatrix}2\\0\\3\end{bmatrix}}="6",(75,8)*+{\begin{bmatrix}2\\1\\4\end{bmatrix}}="7",(75,-8)*+{\begin{bmatrix}1\\2\\4\end{bmatrix}}="8",(75,-24)*+{\begin{bmatrix}0\\2\\3\end{bmatrix}}="9", \ar@{-}^{1}"2";"4"\ar@{-}_{2}"2";"5"\ar@{-}^{3}"4";"6"\ar@{-}_{2}"4";"7"\ar@{-}^{1}"5";"8"\ar@{-}_{3}"5";"9"
\end{xy}.
\end{align*}
  We denote this tree by $\mathrm{Tree}(D)$, and we call $\text{Tree}(D)$ the \emph{intersection vector tree}. 
In this section, we prove the following theorem:
\begin{theorem}\label{thm:main1}
We consider a map
\begin{align*}
    g\colon\ZZ_{\geq0}^3 \to \QQ,\quad \begin{bmatrix}d_1\\d_2\\d_3\end{bmatrix}\mapsto \dfrac{d_1+1}{d_2+1}.
\end{align*}
The Stern-Brocot tree is obtained by replacing each vertex $v$ of $\mathrm{Tree}(D)$ with $g(v)$.
\end{theorem}
For $k\in \{1,2,3\}$, we define the \emph{intersection matrix flip} $\Phi_k$ of $D(L,L_t)$ in direction $k$ as 
\begin{align}
\Phi_k(D(L,L_t))=D(L,\varphi_k(L_t)).
\end{align}
By regarding punctures on the universal cover of $S$ as lattice points on $\RR^2$ with the coordinate axis $\ell_1=\ell_{1;t_0}$ and $\ell_2=\ell_{2;t_0}$, we consider the \emph{gradient} of arcs of $L_t$. We denote the gradient of $\ell$ by $\mathrm{grad}_L(\ell)$. We assume $\mathrm{grad}_L(\ell_{1})=0,\mathrm{grad}_L(\ell_{2})=\infty,\mathrm{grad}_L(\ell_{3})=-1$. 
\begin{example}
We consider an arc $\ell$ in Figure \ref{example}. Then we have 
\begin{align*}
\mathrm{grad}_L(\ell)=\dfrac{3}{2},\quad \text{and} \quad D(L,\ell)=\begin{bmatrix}2\\1\\4\end{bmatrix}.
\end{align*}
\begin{figure}[ht]
    \caption{Arc $\ell$}
    \label{example}
    \vspace{4mm}
\scalebox{0.65}{\begin{tikzpicture}
\node[circle,fill=black,inner sep=2.7](1) at (0,0){}; 
\node[circle,fill=black,inner sep=2.7](2) at (0,-2){}; 
\node[circle,fill=black,inner sep=2.7](3) at (0,-4){}; 
\node[circle,fill=black,inner sep=2.7](5) at (2,0){}; 
\node[circle,fill=black,inner sep=2.7](6) at (2,-2){}; 
\node[circle,fill=black,inner sep=2.7](7) at (2,-4){}; 
\node[circle,fill=black,inner sep=2.7](8) at (4,-2){}; 
\node[circle,fill=black,inner sep=2.7](9) at (4,0){};
\node[circle,fill=black,inner sep=2.7](10) at (4,2){};  
\node[circle,fill=black,inner sep=2.7](11) at (2,2){};  
\draw(1) to (5);
\draw(1) to (3);
\draw(11) to (7);
\draw(3) to (7);
\draw(2) to (6);
\draw(5) to (9);
\draw(6) to (8);
\draw(10) to (8);
\draw(11) to (10);
\draw(1) to (6);
\draw(2) to (7);
\draw(5) to (8);
\draw(11) to (9);
\draw[ultra thick, blue](3) to (10);
\end{tikzpicture}}
\end{figure}
\end{example}
\begin{definition}
For $q\in\QQ\cup\{\infty\}$, if $n(q),d(q)\in\ZZ$ satisfy the following conditions, we say that $\dfrac{n(q)}{d(q)}$ is the \emph{reduced expression} of $q$: 
\begin{itemize}
    \item $q=\dfrac{n(q)}{d(q)}$,
    \vspace{2mm}
    \item $\gcd (n(q),d(q))=1$,
    \vspace{2mm}
    \item $d(q)\geq 0$.
\end{itemize}
Moreover, for a fraction $\dfrac{a}{b}$, if there exists $q\in\QQ\cup\{\infty\}$ such that $\dfrac{a}{b}$ is the reduced expression of $q$, then we say that $\dfrac{a}{b}$ is \emph{irreducible}.
\end{definition}
This expression is determined uniquely. In particular, $\dfrac{0}{1},\dfrac{1}{0}$ are reduced expressions of $0,\infty$ respectively.
\begin{lemma}\label{lem:incsum}
   Let $M=(m_{1},m_{2},m_{3})$ be a triangulation and $\{i,j,k\}=\{1,2,3\}$. The following two conditions are equivalent: 
   \begin{itemize}
   \item[(1)] Either of the following two inequalities holds:
   \begin{align}\label{eq:small-large-assumption}
   \mathrm{grad}_L(m_{i})<\mathrm{grad}_L(m_{j})<\mathrm{grad}_L(m_{k})\quad \text{or}\quad \mathrm{grad}_L(m_{k})<\mathrm{grad}_L(m_{j})<\mathrm{grad}_L(m_{i}).
    \end{align}
  \item[(2)] If $\mathrm{grad}_L(m_{i})=\dfrac{a}{b}$ and $\mathrm{grad}_L(m_{k})=\dfrac{c}{d}$, and they are irreducible fractions, then $\mathrm{grad}_L(m_{j})=\dfrac{a+c}{b+d}$.
   \end{itemize}
  In particular, for any triangulation $M=(m_{1},m_{2},m_{3})$, there exist $a,c\in\ZZ$ and $b,d\in\ZZ_{\geq0}$ such that $\dfrac{a}{b}$ and $\dfrac{c}{d}$ are irreducible and $\left\{\mathrm{grad}_L(m_{1}),\mathrm{grad}_L(m_{2}),\mathrm{grad}_L({m_{3}})\right\}=\left\{\dfrac{a}{b},\dfrac{c}{d},\dfrac{a+c}{b+d}\right\}$.   
\end{lemma}
\begin{proof}
We prove that (1) implies (2). Since $M$ is a triangulation, $m_{i},m_{j},m_{k}$ are as in Figure \ref{fig:tri123} on the universal covering of $S$. When the coordinate of a point shared by 3 arcs is $(0,0)$, the coordinate of the other endpoint of $m_{i}$ is $(b,a)$, and that of $m_{k}$ is $(d,c)$. Therefore, that of $m_{j}$ is $(b+d,a+c)$ and $\mathrm{grad}_L(m_{j})=\dfrac{a+c}{b+d}$. It is clear that (2) implies (1).    
    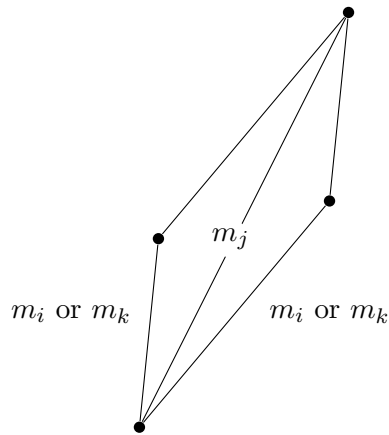
\begin{figure}[ht]
        \caption{Triangulation under the assumption \eqref{eq:small-large-assumption}}
        \label{fig:tri123}
        \vspace{3mm}
\begin{tikzpicture}
\node[circle,fill=black,inner sep=1.5](1) at (0,0){}; 
\node[circle,fill=black,inner sep=1.5](2) at (0.25,2.5){}; 
\node[circle,fill=black,inner sep=1.5](3) at (2.5,3){}; 
\node[circle,fill=black,inner sep=1.5](4) at (2.75,5.5){}; 
\draw (1) to (3);
\draw (3) to (4);
\draw (2) to (4);
\draw (4) to (1);
\draw (1) to (2);
\node[circle,fill=white,inner sep=5](8) at (1.2,2.5){}; 
\node(5) at (2.5,1.5){$m_{i}$ or $m_{k}$};
\node(6) at (-0.9,1.5){$m_{i}$ or $m_{k}$};
\node(7) at (1.2,2.5){$m_{j}$};
\end{tikzpicture}
    \end{figure}
\end{proof}
\begin{remark}\label{irreduciblesum} We note that if (2) in Lamma \ref{lem:incsum} holds, then $\mathrm{grad}_L(m_{j})=\dfrac{a+c}{b+d}$ is irreducible. This fact is shown in the following way: we assume that $\dfrac{a+c}{b+d}$ is not irreducible. Then $m_j$ passes through a lattice point in the section between $(0,0)$ and $(a+c,b+d)$. Since all lattice points are a point on $S$, $m_{i}$ intersects with $m_{j}$ at non-lattice points. This conflicts that $M$ is a triangulation.
Moreover, if $\{\mathrm{grad}_L(m_{1}),\mathrm{grad}_L(m_{2}),\mathrm{grad}_L(m_{3})\}=\left\{\dfrac{a}{b},\dfrac{c}{d},\dfrac{c-a}{d-b}\right\}$ and $\dfrac{a}{b}, \dfrac{c}{d}$ are irreducible, then the reduced expression of $\dfrac{c-a}{d-b}$ is $\dfrac{c-a}{d-b}$ or $\dfrac{a-c}{b-d}$. This fact is proved in the same way as the above.
\end{remark}
Using the above lemma, we obtain the inequality between the gradients.
\begin{lemma}\label{lem:flip-gradient}
  Let $t\in \TT_3$,  $L_t=(\ell_{1;t},\ell_{2;t},\ell_{3;t})$ a triangulation and $\{i,j,k\}=\{1,2,3\}$. We assume that
  \begin{align}\label{eq:small-large-assumption2}
   \mathrm{grad}_L(\ell_{i;t})<\mathrm{grad}_L(\ell_{j;t})<\mathrm{grad}_L(\ell_{k;t}).
   \end{align}
  \begin{itemize}
      \item[(1)] Let $L_{t'}=\{\ell_{i;t},\ell_{j;t'},\ell_{k;t}\}$ be a triangulation of $S$ obtained from $L_t$ by a flip in direction $j$. Then we have 
       \begin{align*}
   \mathrm{grad}_L(\ell_{j;t'})<\mathrm{grad}_L(\ell_{i;t})<\mathrm{grad}_L(\ell_{k;t})\quad  \text{or}\quad  \mathrm{grad}_L(\ell_{i;t})<\mathrm{grad}_L(\ell_{k;t})<\mathrm{grad}_L(\ell_{j;t'}),
    \end{align*}
    \item[(2)] Let $L_{t''}=\{\ell_{i;t''},\ell_{j;t},\ell_{k;t}\}$ be a triangulation of $S$ obtained from $L_t$ by a flip in direction $i$. Then we have 
       \begin{align*}
   \mathrm{grad}_L(\ell_{j;t})<\mathrm{grad}_L(\ell_{i;t''})<\mathrm{grad}_L(\ell_{k;t})
    \end{align*}
     \item[(3)] Let $L_{t'''}=\{\ell_{i;t},\ell_{j;t},\ell_{k;t'''}\}$ be a triangulation of $S$ obtained from $L_t$ by a flip in direction $k$. Then we have 
       \begin{align*}
   \mathrm{grad}_L(\ell_{i;t})<\mathrm{grad}_L(\ell_{k;t'''})<\mathrm{grad}_L(\ell_{j;t}).
    \end{align*}
  \end{itemize}
\end{lemma}
\begin{proof}
We prove (1). By flipping the triangulation in Figure \ref{fig:tri123} in direction $j$, we have a triangulation in Figure \ref{fig:tri123flip}.
  \begin{figure}[ht]
        \caption{Flipped triangulation}
        \label{fig:tri123flip}
        \vspace{3mm}
\begin{tikzpicture}
\node[circle,fill=black,inner sep=1.5](1) at (0,0){}; 
\node[circle,fill=black,inner sep=1.5](2) at (0.25,2.5){}; 
\node[circle,fill=black,inner sep=1.5](3) at (2.5,3){}; 
\node[circle,fill=black,inner sep=1.5](4) at (2.75,5.5){}; 
\draw (1) to (3);
\draw (3) to (4);
\draw (2) to (4);
\draw (2) to (3);
\draw (1) to (2);
\node[circle,fill=white,inner sep=6](8) at (1.2,2.7){}; 
\node(5) at (2,1.5){$m_{i}$};
\node(6) at (-0.4,1.5){$m_{k}$};
\node(7) at (1.2,2.7){$m'_{j}$};
\end{tikzpicture}
    \end{figure}
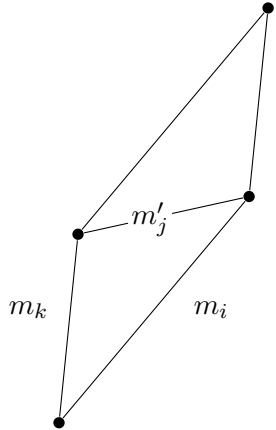
If $\mathrm{grad}_L(\ell_{i;t})=\dfrac{a}{b}$ and $\mathrm{grad}_L(\ell_{k;t})=\dfrac{c}{d}$, then $\mathrm{grad}_L(\ell_{j;t})=\dfrac{c-a}{d-b}$. If the reduced expression of $\mathrm{grad}_L(\ell_{j;t})$ is $\dfrac{c-a}{d-b}$, then by Lemma \ref{lem:incsum} and \eqref{eq:small-large-assumption2}, we have $\mathrm{grad}_L(\ell_{i;t})<\mathrm{grad}_L(\ell_{k;t})<\mathrm{grad}_L(\ell_{j;t'})$. On the other hand, if the reduced expression is $\dfrac{a-c}{b-d}$, then we have $\mathrm{grad}_L(\ell_{j;t'})<\mathrm{grad}_L(\ell_{i;t})<\mathrm{grad}_L(\ell_{k;t})$. The case (2) and (3) are also proved in the same way.
\end{proof}
We note that if $t\in\TT'_3$, then the gradients of arcs of $L_t$ are 0 or more by Lemma \ref{lem:flip-gradient}. Let us consider relation between the gradient and the intersection vector of an arc $\ell$ of $L_t$. The following fact is useful:
\begin{lemma}\label{lem:int-inc}
     Let $\ell\in L_t$ be an edge satisfying $t\in \TT'_3$. Then, $\mathrm{grad}_L(\ell)=\dfrac{a}{b}$ holds and $\dfrac{a}{b}$ is irreducible if and only if $D(L,\ell)=\begin{bmatrix}a-1\\b-1\\a+b-1\end{bmatrix}$ holds.
\end{lemma}
\begin{proof}
We assume that $\mathrm{grad}_L(\ell)=\dfrac{a}{b}$ and $\dfrac{a}{b}$ is irreducible. The line segment from a lattice point to the next lattice point of the line corresponding to $\ell$ passes through $a + b-1$ tiles separated by $\ell_1$ and $\ell_2$. In the process, the line segment intersects the horizontal line $a-1$ times and the vertical line $b-1$ times. Moreover, this line segment intersects the line corresponding to $\ell_3$ once for each passing tile. Therefore, we have $D(L,\ell)=\begin{bmatrix}a-1\\b-1\\a+b-1\end{bmatrix}$. Conversely, assume $ D (L, \ell) = \begin {bmatrix} a-1 \\ b-1 \\ a + b-1 \end {bmatrix}$. We have $\mathrm{grad}_L(\ell) = \dfrac{a}{b}$ by tracing the above discussion backwards. We assume that $\dfrac{a}{b}$ is not irreducible. In this case, the entries of the vector $\begin{bmatrix}a-1\\b-1\\a+b-1\end{bmatrix}$ represent the number of intersections of the line segments through the lattice points of the universal covering. This means that we are circling the same line on the torus more than once, which contradicts the definition of the intersection vector.
\end{proof}

\begin{remark}\label{irreducible}
By Lemma \ref{lem:int-inc}, for $\ell\in L_t$ satisfying $t\in \TT'_3$ and $D(L,\ell_t)=\begin{bmatrix}a\\b\\c\end{bmatrix}$, $\dfrac{a+1}{b+1}$ and $\dfrac{b+1}{a+1}$ are irreducible. 
\end{remark}
We define the \emph{non-middle gradient flip} as a flip that removes an arc whose gradient is the smallest or the largest in the three arcs and adds another arc.
We are ready to prove the first main theorem. 
\begin{proof}[Proof of Theorem \ref{thm:main1}]
Let $t\in\TT_3$. For $L_t=(\ell_{1;t},\ell_{2;t},\ell_{3;t})$, we consider the triple 
\begin{align*}
  \mathrm{grad}_L(L_t)=  (\mathrm{grad}_L(\ell_{1;t}),\mathrm{grad}_L(\ell_{2;t}),\mathrm{grad}_L(\ell_{3;t})),
\end{align*}
where all entries are irreducible.
According to Lemma \ref{lem:int-inc}, a restriction of $g$ to $\{D(L,\ell_{i;t})\}_{t\in\TT'_3,i\in\{1,2,3\}}\setminus\{0\}$ is given by $D(L,\ell)\mapsto \text{(the reduced expression of) } \mathrm{grad}_L(\ell)$. Therefore, it suffices to show that a tree 
\begin{align}\label{gradienttree}
\begin{xy}(40,0)*+{\mathrm{grad}_L(\ell_{3;t_1})}="2",(55,16)*+{\mathrm{grad}_L(\ell_{1;t_2})}="4",(55,-16)*+{\mathrm{grad}_L(\ell_{2;t_3})}="5", 
(77.5,24)*+{\mathrm{grad}_L(\ell_{2;t_4})}="6",(77.5,8)*+{\mathrm{grad}_L(\ell_{3;t_5})}="7",(77.5,-8)*+{\mathrm{grad}_L(\ell_{1;t_6})}="8",(110,28)*+{\mathrm{grad}_L(\ell_{1;t_8})\cdots}="10",(110,20)*+{\mathrm{grad}_L(\ell_{3;t_{9}})\cdots}="11",(110,12)*+{\mathrm{grad}_L(\ell_{1;t_{10}})\cdots}="12",(110,4)*+{\mathrm{grad}_L(\ell_{2;t_{11}})\cdots}="13",(110,-4)*+{\mathrm{grad}_L(\ell_{2;t_{12}})\cdots}="14",(110,-12)*+{\mathrm{grad}_L(\ell_{3;t_{13}})\cdots}="15",(110,-20)*+{\mathrm{grad}_L(\ell_{1;t_{14}})\cdots}="16",(110,-28)*+{\mathrm{grad}_L(\ell_{2;t_{15}})\cdots}="17",
(77.5,-24)*+{\mathrm{grad}_L(\ell_{3;t_7})}="9", \ar@{-}^{1}"2";"4"\ar@{-}_{2}"2";"5"\ar@{-}^{3}"4";"6"\ar@{-}_{2}"4";"7"\ar@{-}^{1}"5";"8"\ar@{-}_{3}"5";"9"\ar@{-}^{1}"6";"10"\ar@{-}_{2}"6";"11"\ar@{-}^{3}"7";"12"\ar@{-}_{1}"7";"13"\ar@{-}^{2}"8";"14"\ar@{-}_{3}"8";"15"\ar@{-}^{1}"9";"16"\ar@{-}_{2}"9";"17"
\end{xy}.
\end{align}
is the Stern-Brocot tree. 
Flips in direction 1 and 2 at $t_1$ are non-middle gradient flips, and we see that flips from left to right in \eqref{subtree} are all non-middle gradient flips inductively by Lemma \ref{lem:flip-gradient}. Furthermore, $\mathrm{grad}_L(\ell_{k;t})$ lying in \eqref{gradienttree} is the second largest number in $\mathrm{grad}_L(L_t)$. Therefore, by Remark \ref{stern-brocot-remark}, it suffices to show 
that a tree
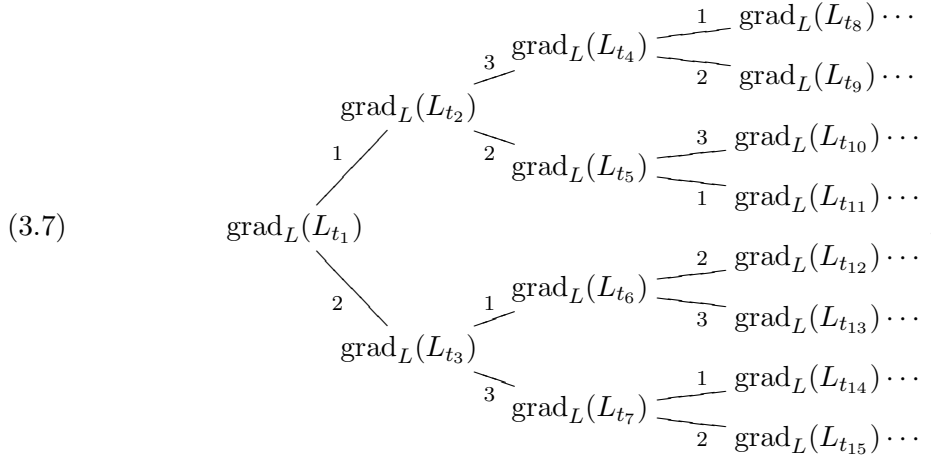
\begin{align}\label{gradienttree2}
\begin{xy}(40,0)*+{\mathrm{grad}_L(L_{t_1})}="2",(55,16)*+{\mathrm{grad}_L(L_{t_2})}="4",(55,-16)*+{\mathrm{grad}_L(L_{t_3})}="5", 
(77.5,24)*+{\mathrm{grad}_L(L_{t_4})}="6",(77.5,8)*+{\mathrm{grad}_L(L_{t_5})}="7",(77.5,-8)*+{\mathrm{grad}_L(L_{t_6})}="8",(110,28)*+{\mathrm{grad}_L(L_{t_8})\cdots}="10",(110,20)*+{\mathrm{grad}_L(L_{t_9})\cdots}="11",(110,12)*+{\mathrm{grad}_L(L_{t_{10}})\cdots}="12",(110,4)*+{\mathrm{grad}_L(L_{t_{11}})\cdots}="13",(110,-4)*+{\mathrm{grad}_L(L_{t_{12}})\cdots}="14",(110,-12)*+{\mathrm{grad}_L(L_{t_{13}})\cdots}="15",(110,-20)*+{\mathrm{grad}_L(L_{t_{14}})\cdots}="16",(110,-28)*+{\mathrm{grad}_L(L_{t_{15}})\cdots}="17",
(77.5,-24)*+{\mathrm{grad}_L(L_{t_7})}="9", \ar@{-}^{1}"2";"4"\ar@{-}_{2}"2";"5"\ar@{-}^{3}"4";"6"\ar@{-}_{2}"4";"7"\ar@{-}^{1}"5";"8"\ar@{-}_{3}"5";"9"\ar@{-}^{1}"6";"10"\ar@{-}_{2}"6";"11"\ar@{-}^{3}"7";"12"\ar@{-}_{1}"7";"13"\ar@{-}^{2}"8";"14"\ar@{-}_{3}"8";"15"\ar@{-}^{1}"9";"16"\ar@{-}_{2}"9";"17"
\end{xy}.
\end{align}
is the Farey triple tree.
We have $\mathrm{grad}_L(L_{t_1})=\left(\dfrac{0}{1},\dfrac{1}{0},\dfrac{1}{1}\right)$. We assume $\mathrm{grad}_L(L_{t})=\left(\dfrac{a}{b},\dfrac{c}{d},\dfrac{e}{f}\right)$. By Lemma \ref{lem:flip-gradient} and Lemma \ref{lem:incsum}, if $\dfrac{a}{b}$ is the smallest or largest in those three, then the edge labeled by 1 is on the right of $\mathrm{grad}_L(L_{t})$ in \eqref{gradienttree2}, and we have $\left(\dfrac{a}{b},\dfrac{c}{d},\dfrac{e}{f}\right)\mapsto\mathrm{grad}_L(L_{t'})=\left(\dfrac{c+e}{d+f},\dfrac{c}{d},\dfrac{e}{f}\right)$ by a flip in direction 1 (we note that all of $\dfrac{c+e}{d+f},\dfrac{c}{d},\dfrac{e}{f}$ are irreducible again by Remark \ref{irreduciblesum}). Similarly, if $\dfrac{c}{d}$ or $\dfrac{e}{f}$ is the smallest or largest in those three, we have the desired triple. Therefore, \eqref{gradienttree2} corresponds with the Farey triple tree, and this finishes the proof.
\end{proof}

\begin{remark}\label{contrarymain1}
Contrary to Theorem \ref{thm:main1}, given a positive irreducible fraction, it is also possible to give the corresponding intersection vector: $\begin{bmatrix}
a-1\\b-1\\a+b-1
\end{bmatrix}$ for $\dfrac{a}{b}$.
\end{remark}

For the sake of discussion in Section 4, we give an description of intersection matrices.
\begin{corollary}\label{cor:vi--viii}
Let $t\in \TT'_3$ and $L_t=(\ell_{1;t},\ell_{2;t},\ell_{3;t})$ a triangulation. The intersection matrix $D(L, L_t)=(d_{ij})$ satisfies just one of the following:
       \vspace{2mm}
       \begin{itemize}
        \item [(i)] $\begin{bmatrix} d_{11}&d_{12}&d_{11}+d_{12}+1\\d_{21}&d_{22}&d_{21}+d_{22}+1\\d_{11}+d_{21}+1&d_{12}+d_{22}+1&d_{11}+d_{12}+d_{21}+d_{22}+3\end{bmatrix}$
        \vspace{2mm}
         \item [(ii)] $\begin{bmatrix} d_{12}+d_{13}+1&d_{12}&d_{13}\\d_{22}+d_{23}+1&d_{22}&d_{23}\\d_{12}+d_{13}+d_{22}+d_{23}+3&d_{12}+d_{22}+1&d_{13}+d_{23}+1\end{bmatrix}$
        \vspace{2mm}
         \item [(iii)] $\begin{bmatrix} d_{11}&d_{11}+d_{13}+1&d_{13}\\d_{21}&d_{21}+d_{23}+1&d_{23}\\d_{11}+d_{21}+1&d_{11}+d_{21}+d_{13}+d_{23}+3&d_{13}+d_{23}+1\end{bmatrix}$
        \end{itemize}
Moreover, we have the following diagram:
\begin{align*}
\begin{xy}(30,0)*+{\mathrm{(iii)}}="6",(60,0)*+{\mathrm{(ii)}}="7",(90,0)*+{\mathrm{(i)}}="8",
\ar@<0.5ex>^1"6";"7"\ar@<0.5ex>^2"7";"6"\ar@<0.5ex>^3"7";"8"\ar@<0.5ex>^1"8";"7"\ar@<1.5ex>@/^6mm/^3"6";"8"\ar@<-0.5ex>@/_6mm/^2"8";"6"
\end{xy},
\end{align*}
where $\begin{xy}(0,0)*+{\mathrm{(n)}}="1",(15,0)*+{\mathrm{(m)}}="2",\ar@{->}^{k}"1";"2"\end{xy}$ implies that $D(L,L_t)$ satisfying $\mathrm{(m)}$ is obtained from $\mathrm{(n)}$ by $\Phi_k$.        
\end{corollary}
\begin{proof}
The statement follows from the fact that $D(L,L_{t_1})$ satisfies $\mathrm{(i)}$, and Lemmas \ref{lem:int-inc} and \ref{lem:incsum}.
\end{proof}
\section{Initial intersection vector tree and Calkin-Wilf tree}
In contrast to the previous section, we correspond the intersection vectors to vertices of a subtree $\TT'_3$ of \eqref{subtree} as 
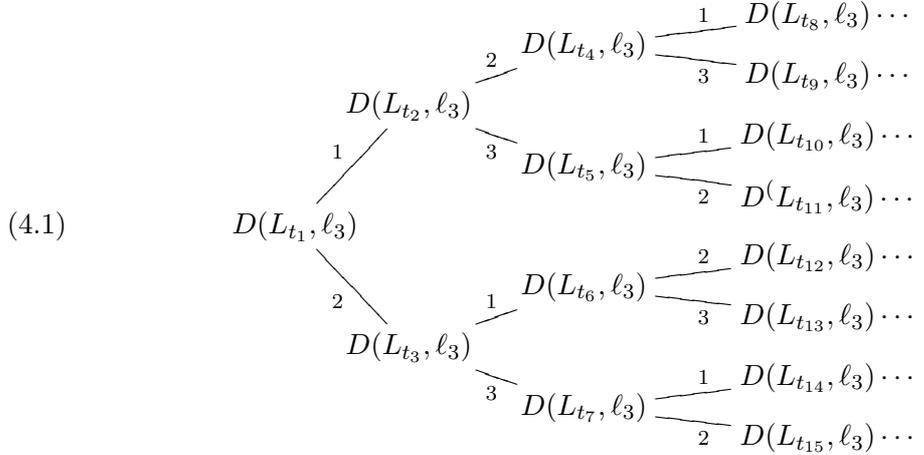
\begin{align}\label{tree-dual}
\begin{xy}(40,0)*+{D(L_{t_1},\ell_{3})}="2",(55,16)*+{D(L_{t_2},\ell_{3})}="4",(55,-16)*+{D(L_{t_3},\ell_{3})}="5", 
(78,24)*+{D(L_{t_4},\ell_{3})}="6",(78,8)*+{D(L_{t_5},\ell_{3})}="7",(78,-8)*+{D(L_{t_6},\ell_{3})}="8",(110,28)*+{D(L_{t_8},\ell_{3})\cdots}="10",(110,20)*+{D(L_{t_9},\ell_{3})\cdots}="11",(110,12)*+{D(L_{t_{10}},\ell_{3})\cdots}="12",(110,4)*+{D^(L_{t_{11}},\ell_{3})\cdots}="13",(110,-4)*+{D(L_{t_{12}},\ell_{3})\cdots}="14",(110,-12)*+{D(L_{t_{13}},\ell_{3})\cdots}="15",(110,-20)*+{D(L_{t_{14}},\ell_{3})\cdots}="16",(110,-28)*+{D(L_{t_{15}},\ell_{3})\cdots}="17",
(78,-24)*+{D(L_{t_7},\ell_{3})}="9", \ar@{-}^{1}"2";"4"\ar@{-}_{2}"2";"5"\ar@{-}^{2}"4";"6"\ar@{-}_{3}"4";"7"\ar@{-}^{1}"5";"8"\ar@{-}_{3}"5";"9"\ar@{-}^{1}"6";"10"\ar@{-}_{3}"6";"11"\ar@{-}^{1}"7";"12"\ar@{-}_{2}"7";"13"\ar@{-}^{2}"8";"14"\ar@{-}_{3}"8";"15"\ar@{-}^{1}"9";"16"\ar@{-}_{2}"9";"17"
\end{xy}.
\end{align}
That is, we assign $D(L_{t_{i}},\ell_{3})$ to $t_i$. Its first seven vertices are as follows:
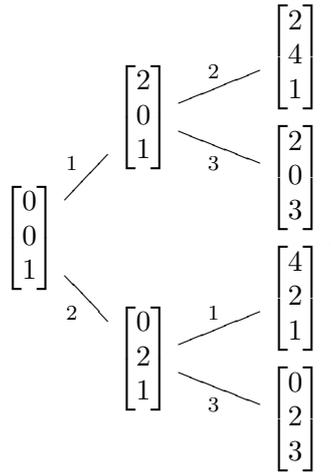
\begin{align*}
\begin{xy}(40,0)*+{\begin{bmatrix}0\\0\\1\end{bmatrix}}="2",(55,16)*+{\begin{bmatrix}2\\0\\1\end{bmatrix}}="4",(55,-16)*+{\begin{bmatrix}0\\2\\1\end{bmatrix}}="5", 
(75,24)*+{\begin{bmatrix}2\\4\\1\end{bmatrix}}="6",(75,8)*+{\begin{bmatrix}2\\0\\3\end{bmatrix}}="7",(75,-8)*+{\begin{bmatrix}4\\2\\1\end{bmatrix}}="8",(75,-24)*+{\begin{bmatrix}0\\2\\3\end{bmatrix}}="9", \ar@{-}^{1}"2";"4"\ar@{-}_{2}"2";"5"\ar@{-}^{2}"4";"6"\ar@{-}_{3}"4";"7"\ar@{-}^{1}"5";"8"\ar@{-}_{3}"5";"9"
\end{xy}.
\end{align*}
We denote this tree by $\mathrm{Tree} (D^\dag)$ and we call it the \emph{initial intersection vector tree}.
We make a preparation for describing the main theorem of this section. In this section, we regard $L_t$ as the initial triangulation. We define the \emph{initial intersection matrix flip}\footnote{In the context of cluster algebra theory, this operation is the special case of an \emph{initial mutation} of an $D$-matrix. See \cite{rs}.} $\Psi_k$ of $D(L_t,L)$ in direction $k$ as 
\begin{align}
\Psi_k(D(L_t,L))=D(\varphi_k(L_t),L).
\end{align}
The following proposition is clear:
\begin{proposition}\label{prop:duality}
We have 
\begin{align*}
    D(L_t,L)=(D(L,L_t))^T,
\end{align*}
where $T$ is the transposition.
\end{proposition}
By using this duality, we have the following property:
\begin{proposition}\label{prop:maximality}
For $\begin{xy}(0,0)*+{t}="A",(10,0)*+{t'}="B",\ar@{-}^k"A";"B" \end{xy}\in \TT'_3$, let
\begin{align*}
 D(L_t,\ell_3)=\begin{bmatrix}d_1\\d_2\\d_3\end{bmatrix}\quad \text{and} \quad  D(L_{t'},\ell_3)=\begin{bmatrix}d_1'\\d_2'\\d_3'\end{bmatrix}.
 \end{align*}
If $i\neq k$, then we have $d_i=d'_i$.
Furthermore, when $D(L_t,\ell_3)\geq 0$ and $D(L_{t'},\ell_3)\geq 0$, $d_k\neq \max\{d_1,d_2,d_3\}$ if and only if  $d'_k=\max\{d'_1,d'_2,d'_3\}$. 
\end{proposition}
\begin{proof}
The former follows from Proposition \ref{prop:duality}. We prove the ``only if" part of the latter. If $d_k=\max\{d_1,d_2,d_3\}$, the gradient of an arc corresponding to $k$th row of $D(L,L_t)$ is the second largest in three arcs because of Proposition \ref{prop:duality} and Lemmas \ref{lem:int-inc}, \ref{lem:incsum}. By Lemma \ref{lem:flip-gradient}, the gradient of an arc corresponding to $k$th row of $D(L,L_{t'})$ is not the second largest in three arcs. By Lemma Proposition \ref{prop:duality} and Lemmas \ref{lem:int-inc}, \ref{lem:incsum} again, we have $d'_k\neq\max\{d'_1,d'_2,d'_3\}$.  We prove the ``if" part. We assume that $d'_k=0$. By Proposition \ref{prop:duality} and Lemma \ref{lem:flip-gradient}, we have $D(L,L_{t'})\leq 0$ and this is confliction. In the case of $d_k\neq 0$, it is proved by considering the inverse of the ``only if" part with $d'_k\neq0$. 
\end{proof}
By Proposition \ref{prop:maximality}, if ${D(L_t,\ell_{3})}$ lies on the right endpoint of an edge labeled by $k$ in the tree of \eqref{tree-dual}, then the $k$th element of $D(L_t,\ell_{3})$ is the maximal in those three. In the rest of this section, we prove the following theorem:
\begin{theorem}\label{thm:main2}
 We set $D(L_t,\ell_3)=\begin{bmatrix}d_{13;t}\\d_{23;t}\\d_{33;t}\end{bmatrix}$. We define a map
 \begin{align*}
    h\colon \{D(L_t,\ell_3)\}_{t\in \TT'_3}\to \QQ
 \end{align*} 
 inductively as follows: we assign \begin{align*}
     D(L_{t_1},\ell_3)\mapsto\dfrac{d_{13;t_1}+1}{d_{23;t_1}+1}=\dfrac{1}{1}.
     \end{align*}
     Let $\{a,b,c\}=\{1,2,3\}$. When $D(L_t,\ell_3)\mapsto \dfrac{d_{a3;t}+1}{d_{b3;t}+1}$, and $\begin{xy}(0,0)*+{D(L_t,\ell_3)}="A",(23,0)*+{D(L_{t'},\ell_3)}="B",\ar@{-}^k"A";"B" \end{xy}$ with $d(t_1,t)<d(t_1,t')$,
\begin{itemize}
    \item if $k=a$, then we assign $D(L_{t'},\ell_3)\mapsto \dfrac{d_{c3;t}+1}{d_{b3;t}+1}$,
    \vspace{2mm}
    \item if $k=b$, then we assign $D(L_{t'},\ell_3)\mapsto \dfrac{d_{a3;t}+1}{d_{c3;t}+1}$.
    \end{itemize}
The Calkin-Wilf tree is obtained by replacing each vertex $v$ of $\mathrm{Tree}(D^\dag)$ with $h(v)$.
\end{theorem}
In the rest of this section, we prove Theorem \ref{thm:main2}. The following lemma is duality of Lemma \ref{cor:vi--viii}:
\begin{lemma}\label{lem:description-dual}
We fix $t\in \TT'_3$. The intersection matrix $D(L_t, L)=(d_{ij})$ satisfies just one of the following: 
    \begin{itemize}
        \item [(i)] $\begin{bmatrix} d_{11}&d_{12}&d_{11}+d_{12}+1\\d_{21}&d_{22}&d_{21}+d_{22}+1\\d_{11}+d_{21}+1&d_{12}+d_{22}+1&d_{11}+d_{12}+d_{21}+d_{22}+3\end{bmatrix}$
        \vspace{2mm}
         \item [(ii)] $\begin{bmatrix} d_{21}+d_{31}+1&d_{22}+d_{32}+1&d_{21}+d_{31}+d_{22}+d_{32}+3\\d_{21}&d_{22}&d_{21}+d_{22}+1\\d_{31}&d_{32}&d_{31}+d_{32}+1\end{bmatrix}$
        \vspace{2mm}
         \item [(iii)] $\begin{bmatrix} d_{11}&d_{12}&d_{11}+d_{12}+1\\d_{11}+d_{31}+1&d_{12}+d_{32}+1&d_{11}+d_{12}+d_{31}+d_{32}+3\\d_{31}&d_{32}&d_{31}+d_{32}+1\end{bmatrix}$.
\end{itemize}
Moreover, we have the following diagram:
\begin{align*}
\begin{xy}(30,0)*+{\mathrm{(iii)}}="6",(60,0)*+{\mathrm{(ii)}}="7",(90,0)*+{\mathrm{(i)}}="8",
\ar@<0.5ex>^1"6";"7"\ar@<0.5ex>^2"7";"6"\ar@<0.5ex>^3"7";"8"\ar@<0.5ex>^1"8";"7"\ar@<1.5ex>@/^6mm/^3"6";"8"\ar@<-0.5ex>@/_6mm/^2"8";"6"
\end{xy},
\end{align*}
where $\begin{xy}(0,0)*+{\mathrm{(n)}}="1",(15,0)*+{\mathrm{(m)}}="2",\ar@{->}^{k}"1";"2"\end{xy}$ implies that $D(L_t,L)$ satisfying $\mathrm{(m)}$ is obtained from $\mathrm{(n)}$ by $\Psi_k$.  
\end{lemma}
\begin{proof}
It follows from Corollary \ref{cor:vi--viii} and Proposition \ref{prop:duality}.
\end{proof}
Let us prove the main theorem in this section.
\begin{proof}[Proof of Theorem \ref{thm:main2}]
 For a fraction $q$, we denote the numerator and denominator of $q$ by $q_n, q_d$, respectively. It suffices to show the following:
for any (n) in (i)--(iii) in the diagram of Lemma \ref{lem:description-dual} and $D(L_t,\ell_3)$ satisfying (n), if
\begin{align*}
h(D(L_t,\ell_3))_n=x \quad h(D(L_t,\ell_3))_d=y,
\end{align*}
then we have 
\begin{align*}
    &h(\Psi_a(D(L_{t},\ell_3)))_n=x+y, \quad h(\Psi_a(D(L_{t},\ell_3)))_d=y,\\
 &h(\Psi_b(D(L_{t},\ell_3)))_n=x, \quad h(\Psi_b(D(L_{t},\ell_3)))_d=x+y.
\end{align*}
We set $D(L_t,L)=\begin{bmatrix}d_{11;t}&d_{12;t}&d_{13;t}\\d_{21;t}&d_{22;t}&d_{23;t}\\d_{31;t}&d_{32;t}&d_{33;t}\end{bmatrix}$. We prove the case that $D(L_t,\ell_3)$ satisfies (i). First, we assume that 
 \begin{align*}
     h(D(L_t,\ell_3))=\frac{d_{13;t}+1}{d_{23;t}+1}.
 \end{align*}
By Lemma \ref{lem:description-dual} and definition of $h$, for $\begin{xy}(0,0)*+{t}="A",(10,0)*+{t'}="B",\ar@{-}^1"A";"B" \end{xy}$ and $\begin{xy}(0,0)*+{t}="A",(10,0)*+{t''}="B",\ar@{-}^2"A";"B" \end{xy}$, we have
 \begin{align*}
    h(D(L_t,\ell_3))&=\frac{d_{13;t}+1}{d_{23;t}+1}=\dfrac{d_{11;t}+d_{12;t}+2}{d_{21;t}+d_{22;t}+2},\\[2pt]
    h(D(L_{t'},\ell_3))&=\frac{d_{33;t}+1}{d_{23;t}+1}=\dfrac{d_{11;t}+d_{12;t}+d_{21;t}+d_{22;t}+4}{d_{21;t}+d_{22;t}+2}\\[2pt]
    &=\dfrac{(d_{11;t}+d_{12;t}+2)+(d_{21;t}+d_{22;t}+2)}{d_{21;t}+d_{22;t}+2},\\[2pt]
    h(D(L_{t''},\ell_3))&=\dfrac{d_{13;t}+1}{d_{33;t}+1}=\dfrac{d_{11;t}+d_{12;t}+2}{d_{11;t}+d_{12;t}+d_{21;t}+d_{22;t}+4}\\[2pt]
   & =\dfrac{d_{11;t}+d_{12;t}+2}{(d_{11;t}+d_{12;t}+2)+(d_{21;t}+d_{22;t}+2)}.
\end{align*}
Moreover, we have $h(D(L_{t'},\ell_3))=\dfrac{d_{33;t'}+1}{d_{23;t'}+1}$ and $h(D(L_{t''},\ell_3))=\dfrac{d_{13;t''}+1}{d_{33;t''}+1}$ by Proposition \ref{prop:maximality}.
 Second, in the case that 
 \begin{align*}
     h(D(L_t,\ell_3))=\frac{d_{23;t}+1}{d_{13;t}+1},
 \end{align*}
 we can prove in the same way as the first case.
 Therefore, in (i), $D(L_t,\ell_3)$ satisfies the desired condition. 
 By symmetry, we can also prove the case that $D(L_t,\ell_3)$ satisfies (ii) or (iii).
\end{proof}
\begin{remark}\label{contrarymain2}
Contrary to Theorem \ref{thm:main2}, given a positive irreducible fraction, it is also possible to give the corresponding initial intersection vector inductively by using Calkin-Wilf tree. Actually, the entry of $D(L,\ell)$ that is not used in the map $h$ in Theorem \ref{thm:main2} can be recovered as $a+b+1$ using the other two components $a$ and $b$ from Lemma \ref{lem:description-dual}.
\end{remark}
\section{Christoffel tree and Cohn tree}
In this section, we give the duality of the Christoffel tree and the Cohn tree by using the duality of the Stern-Brocot tree and the Calkin-Wilf tree. First, we define the Christoffel path\footnote{This is also called the \emph{maximal Dyck path.}} and the Christoffel word. Let $\ell_{y/x}$ be a line segment in $\mathbb{R}^2$ from $(0,0)$ to $(x,y)\in\ZZ^2$. Now, we assume that $y/x$ is irreducible. We define the \emph{(lower) Christoffel path} $c_{y/x}$ of slope $y/x$ as the staircase walk on $\ZZ_{\geq0}^2$ from $(0,0)$ to $(x,y)$ satisfying the following conditions:
\begin{itemize}
    \item $c_{y/x}$ does not pass through the region above $\ell_{y/x}$,
    \item no element in $\ZZ^2$ is contained in the interior of the region enclosed by $c_{y/x}$ and $\ell_{y/x}$. 
\end{itemize}
Next, we consider a word that is formed by adding $a$ to the right for each parallel move and adding $b$ to the right for each perpendicular move from $(0,0)$ to $(x,y)$ along the Christoffel path $c_{y/x}$. This word is called the \emph{Christoffel word} of slope $y/x$.
\begin{example}
When $(x,y)=(5,3)$, the Christoffel path is the staircase walk in the Figure \ref{fig:Dyck-path}. Moreover, the Christoffel word of slope $3/5$ is $aabaabab$.
\begin{figure}[ht]
\caption{Christoffel path of slope $3/5$.}
\vspace{2mm}
\begin{center}
$\begin{tikzpicture}
\draw[dashed, step=1,color=black] (0,0) grid (5,3);
\draw[line width=2,color=black] (0,0)--(2,0)--(2,1)--(3,1)--(4,1)--(4,2)--(5,2)--(5,3);
\draw(0,0)--(5,3);
\draw (2.5,2.5) node[anchor=north] {$\ell_{3/5}$};
\draw (3.5,0.8) node[anchor=north] {$c_{3/5}$};
\end{tikzpicture}$
\label{fig:Dyck-path}
\end{center}
\end{figure}
\end{example}
We use $\prec$ as the lexicographic order. We define the Christoffel tree. This is the following tree: the root is $(a,b)$, and the generation rule is that a parent $(u,v)$ has the following children:

\begin{align*}
\begin{xy}(0,0)*+{(u,v)}="1",(-15,-15)*+{(u,uv)}="2",(15,-15)*+{(uv,v)}="3", \ar@{-}"1";"2"\ar@{-}"1";"3"
\end{xy}.
\end{align*}
The first few terms are as follows:
\begin{align*}
\begin{xy}(40,0)*+{(a,b)}="2",(55,16)*+{(ab,b)}="4",(55,-16)*+{(a,ab)}="5", 
(80,24)*+{(ab^2,b)}="6",(80,8)*+{(ab,ab^2)}="7",(80,-8)*+{(a^2b,ab)}="8",(115,28)*+{(ab^3,b)\cdots}="10",(115,20)*+{(ab^2,ab^3)\cdots}="11",(115,12)*+{(abab^2,ab^2)\cdots}="12",(115,4)*+{(ab,abab^2)\cdots}="13",(115,-4)*+{(a^2bab,ab)\cdots}="14",(115,-12)*+{(a^2b,a^2bab)\cdots}="15",(115,-20)*+{(a^3b,a^2b)\cdots}="16",(115,-28)*+{(a,a^3b)\cdots}="17",
(80,-24)*+{(a,a^2b)}="9", \ar@{-}^1"2";"4"\ar@{-}_2"2";"5"\ar@{-}^3"4";"6"\ar@{-}_2"4";"7"\ar@{-}^1"5";"8"\ar@{-}_3"5";"9"\ar@{-}^1"6";"10"\ar@{-}_2"6";"11"\ar@{-}^3"7";"12"\ar@{-}_1"7";"13"\ar@{-}^2"8";"14"\ar@{-}_3"8";"15"\ar@{-}^1"9";"16"\ar@{-}_2"9";"17"
\end{xy}.
\end{align*}
Here, the labeling rule is the same as the Calkin-Wilf tree's rule. We denote it by $\mathrm{Tree}(\mathrm{Ch})$. 
The Christoffel tree is also constructed in the following way: first, we consider the \emph{Christoffel triple tree}. This is a full binary tree given in the following way: the root is $\left(a,b,ab\right)$, and the generation rule is that a parent $\left(u,v,w\right)$ has the following two children: if the second largest word in lexicographic order is (i) $u$, (ii) $v$, (iii)$w$, then 
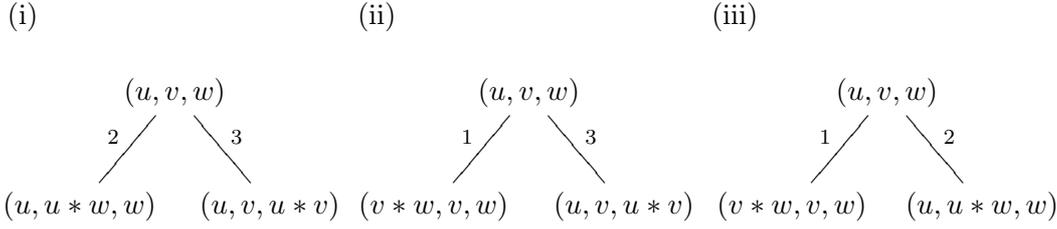
\begin{align*}
\begin{xy}(-20,10)*+{\mathrm{(i)}}="0",(0,0)*+{(u,v,w)}="1",(-12.5,-15)*+{(u,u*w,w)}="2",(12.5,-15)*+{(u,v,u*v)}="3", \ar@{-}_2"1";"2"\ar@{-}^3"1";"3"
\end{xy}
\begin{xy}(-20,10)*+{\mathrm{(ii)}},(0,0)*+{(u,v,w)}="1",(-12.5,-15)*+{(v*w,v,w)}="2",(12.5,-15)*+{(u,v,u*v)}="3", \ar@{-}_1"1";"2"\ar@{-}^3"1";"3"
\end{xy}
\begin{xy}(-20,10)*+{\mathrm{(iii)}},(0,0)*+{(u,v,w)}="1",(-12.5,-15)*+{(v*w,v,w)}="2",(12.5,-15)*+{(u,u*w,w)}="3",(25,-18)*+{.}, \ar@{-}_1"1";"2"\ar@{-}^2"1";"3"
\end{xy}
\end{align*}
where 
\begin{align*}
    u*v=\begin{cases}uv & \text{if } u\prec v\\
    vu & \text{if } v\prec u.
    \end{cases}
\end{align*}
We denote it by $\mathrm{Tree}(3\mathrm{Ch})$. 
The Christoffel tree is obtained from the Christoffel triple tree by replacing vertices with pairs of largest and smallest words in them.  
The following theorem is proved by Berstel and de Luca.
\begin{theorem}[\cite{BdL}]\label{Christoffel-main}
\noindent
\begin{itemize}
\item[(1)] In the Christoffel tree, all elements in each vertex are Christoffel words. Furthermore, for $(u,v)\in \mathrm{Tree}(\mathrm{Ch})$, $uv$ is a Christoffel word.
\item[(2)] The tree obtained by replacing $(u,v)\in \mathrm{Tree}(\mathrm{Ch})$ with $\dfrac{|u|}{|v|}$ is the Calkin-Wilf tree.
\item[(3)] The tree obtained by replacing $(u,v)\in \mathrm{Tree}(\mathrm{Ch})$ with $\dfrac{|uv|_b}{|uv|_a}$ is the Stern-Brocot tree.
\end{itemize}
Here, $|u|$ is the total number of letters $a$ and $b$ in $u$, $|u|_a$ is the number of letters $a$, $|u|_b$ is the number of letters $b$.
\end{theorem}

Next, we give the duality of Theorem \ref{Christoffel-main}.

We introduce the Cohn tree. First, we define the Cohn triple tree, whose vertices are triples of words. The root is $(a,b,ab)$, and generation rule is as follows:
for $(u,v,w)$, its left children is $(u_1,v_1,w_1)$, where $u_1,v_1,w_1$ is obtained from $(u,v,w)$ by replacing all $b$ in their words with $ab$ respectively. On the other hand, its right children is $(u_2,v_2,w_2)$, where $u_2,v_2,w_2$ is obtained from $(u,v,w)$ by replacing all $a$ in their words with $ab$ respectively.
The first few terms are as follows:
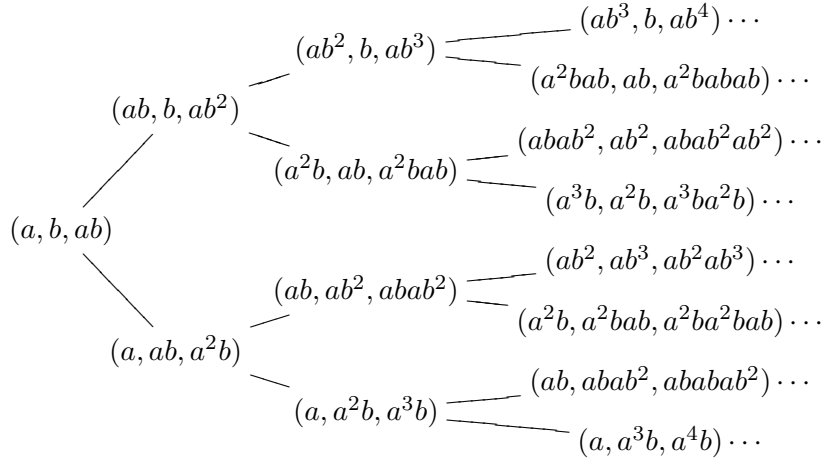
\begin{align*}
\begin{xy}(40,0)*+{(a,b,ab)}="2",(55,16)*+{(ab,b,ab^2)}="4",(55,-16)*+{(a,ab,a^2b)}="5", 
(80,24)*+{(ab^2,b,ab^3)}="6",(80,8)*+{(a^2b,ab,a^2bab)}="7",(80,-8)*+{(ab,ab^2,abab^2)}="8",(120,28)*+{(ab^3,b,ab^4)\cdots}="10",(120,20)*+{(a^2bab,ab,a^2babab)\cdots}="11",(120,12)*+{(abab^2,ab^2,abab^2ab^2)\cdots}="12",(120,4)*+{(a^3b,a^2b,a^3ba^2b)\cdots}="13",(120,-4)*+{(ab^2,ab^3,ab^2ab^3)\cdots}="14",(120,-12)*+{(a^2b,a^2bab,a^2ba^2bab)\cdots}="15",(120,-20)*+{(ab,abab^2,ababab^2)\cdots}="16",(120,-28)*+{(a,a^3b,a^4b)\cdots}="17",
(80,-24)*+{(a,a^2b,a^3b)}="9", \ar@{-}"2";"4"\ar@{-}"2";"5"\ar@{-}"4";"6"\ar@{-}"4";"7"\ar@{-}"5";"8"\ar@{-}"5";"9"\ar@{-}"6";"10"\ar@{-}"6";"11"\ar@{-}"7";"12"\ar@{-}"7";"13"\ar@{-}"8";"14"\ar@{-}"8";"15"\ar@{-}"9";"16"\ar@{-}"9";"17"
\end{xy}.
\end{align*}
By definition, the third entry of each vertex is a word combined the first entry with the second entry. We call this tree the \emph{Cohn triple tree} and denote it by $\mathrm{Tree}(3\mathrm{Co})$. Furthermore, we define the \emph{Cohn tree} as a tree obtained by replacing each vertex in the Cohn triple tree with its first and second entries. The first few terms of the Cohn tree are as follows:
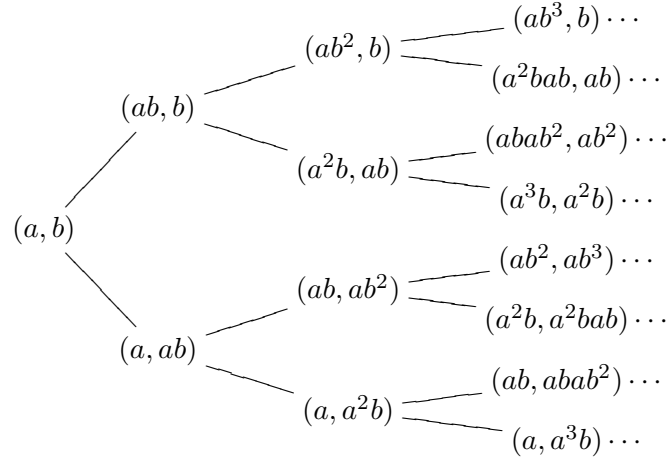
\begin{align*}
\begin{xy}(40,0)*+{(a,b)}="2",(55,16)*+{(ab,b)}="4",(55,-16)*+{(a,ab)}="5", 
(80,24)*+{(ab^2,b)}="6",(80,8)*+{(a^2b,ab)}="7",(80,-8)*+{(ab,ab^2)}="8",(110,28)*+{(ab^3,b)\cdots}="10",(110,20)*+{(a^2bab,ab)\cdots}="11",(110,12)*+{(abab^2,ab^2)\cdots}="12",(110,4)*+{(a^3b,a^2b)\cdots}="13",(110,-4)*+{(ab^2,ab^3)\cdots}="14",(110,-12)*+{(a^2b,a^2bab)\cdots}="15",(110,-20)*+{(ab,abab^2)\cdots}="16",(110,-28)*+{(a,a^3b)\cdots}="17",
(80,-24)*+{(a,a^2b)}="9", \ar@{-}"2";"4"\ar@{-}"2";"5"\ar@{-}"4";"6"\ar@{-}"4";"7"\ar@{-}"5";"8"\ar@{-}"5";"9"\ar@{-}"6";"10"\ar@{-}"6";"11"\ar@{-}"7";"12"\ar@{-}"7";"13"\ar@{-}"8";"14"\ar@{-}"8";"15"\ar@{-}"9";"16"\ar@{-}"9";"17"
\end{xy}.
\end{align*}
We denote this tree by $\mathrm{Tree}(\mathrm{Co})$. 
We also introduce the \emph{combined Cohn tree}.
This tree is obtained from the Cohn tree by replacing each vertex with the third element of it. We denote it by $\mathrm{Tree}(\mathrm{cCo})$.
We prove the following property:
\begin{theorem}[(1):\cite{aig}*{Theorem 7.6}]\label{dualChristoffel-main}
\noindent
\begin{itemize}
\item[(1)] In the Cohn tree, all elements in each vertex are Christoffel words. Furthermore, for $(u,v)\in \mathrm{Tree}(\mathrm{Co})$, $uv$ is a Christoffel word.
\item[(2)] The tree obtained by replacing $(u,v)\in \mathrm{Tree}(\mathrm{Co})$ with $\dfrac{|u|}{|v|}$ is the Stern-Brocot tree.
\item[(3)] The tree obtained by replacing $(u,v)\in \mathrm{Tree}(\mathrm{Co})$ with $\dfrac{|uv|_b}{|uv|_a}$ is the Calkin-Wilf tree.
\end{itemize}
\end{theorem}
These are dual statements of Theorem \ref{Christoffel-main}. We only prove (2) and (3).
\begin{proof}
    First, we prove (3). For the root $ab$ in $\mathrm{Tree}(\mathrm{cCo})$, we have $\dfrac{|ab|_b}{|ab|_a}=\dfrac{1}{1}$.
For $w\in \mathrm{Tree}(\mathrm{cCo})$, we assume that $w$ has the following two children:
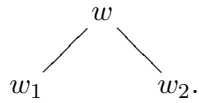
\begin{align*}
\begin{xy}(0,0)*+{w}="1",(-10,-10)*+{w_1}="2",(10,-10)*+{w_2.}="3", \ar@{-}"1";"2"\ar@{-}"1";"3" 
\end{xy}
\end{align*}
Then it suffices to prove that
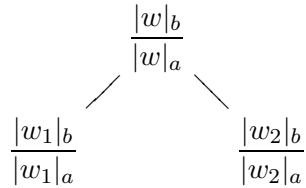
\begin{align*}
\begin{xy}(0,0)*+{\dfrac{|w|_b}{|w|_a}}="1",(-15,-15)*+{\dfrac{|w_1|_b}{|w_1|_a}}="2",(15,-15)*+{\dfrac{|w_2|_b}{|w_2|_a}}="3", \ar@{-}"1";"2"\ar@{-}"1";"3" 
\end{xy}
\end{align*}
is the Calkin-Wilf tree's generation rule, that is,
\begin{align*}
|w_1|_b=|w|_b,\quad &|w_1|_a=|w|_a+|w|_b,\\
|w_2|_b=|w|_a+|w|_b, \quad &|w_2|_a=|w|_a.
\end{align*}
It follows form the generation rule of $\mathrm{Tree}(3\mathrm{Co})$.

Next, we prove (2). For $t\in \TT'_3$ and corresponding vertex $(u_t,v_t,u_tv_t)\in \mathrm{Tree}(3\mathrm{Co})$, we consider 
\begin{align}
    \vv_t=\begin{bmatrix}
    |u_t|-1\\ 
    |v_t|-1\\ 
    |u_tv_t|-1\\
    \end{bmatrix}.
\end{align}
We denote the tree obtained from $\TT'_3$ by replacing $t$ with $\vv_t$ by $\mathrm{Tree}(V)$.
By Theorem \ref{thm:main1}, it suffices to show that $\mathrm{Tree}(V)=\mathrm{Tree}(D)$.
For $t\in\TT'_3$, we consider the following matrix $W_t$: when $t=t_1$, we set
\begin{align*}
    W_{t_1}=\begin{bmatrix}
    |a|_b-1 & |b|_b-1 & |ab|_b-1\\
    |a|_a-1 & |b|_a-1 & |ab|_a-1\\
    |a|-1 & |b|-1 & |ab|-1
    \end{bmatrix}
    =\begin{bmatrix}
    -1 & 0 & 0\\
    0 & -1 & 0\\
    0 & 0 & 1
    \end{bmatrix}.
\end{align*}
Furthermore, for $t\in\TT'_3$, we define $W_t$ respectively in the following way:  $\begin{xy}(0,0)*+{t}="A",(10,0)*+{t'}="B",\ar@{-}^1"A";"B" \end{xy}$ with $d(t_1,t)<d(t_1,t')$ and 
\begin{align*}
    W_{t}=\begin{bmatrix}
    |u_t|_x-1 & |v_t|_x-1 & |u_tv_t|_x-1\\
    |u_t|_y-1 & |v_t|_y-1 & |u_tv_t|_y-1\\
    |u_t|-1 & |v_t|-1 & |u_tv_t|-1\\
\end{bmatrix},
\end{align*}
then we set
\begin{align*}
    W_{t'}=\begin{bmatrix}
    |u_{t'}|-1 & |v_{t'}|-1 & |u_{t'}v_{t'}|-1\\
    |u_{t'}|_y-1 & |v_{t'}|_y-1 & |u_{t'}v_{t'}|_y-1\\
    |u_{t'}|_x-1 & |v_{t'}|_x-1 & |u_{t'}v_{t'}|_x-1\\
\end{bmatrix},
\end{align*}
where $\{x,y\}=\{a,b\}$.
In the same way, if
$\begin{xy}(0,0)*+{t}="A",(10,0)*+{t'}="B",\ar@{-}^1"A";"B" \end{xy}$ with $d(t_1,t)<d(t_1,t')$ and 
\begin{align*}
    W_{t}=\begin{bmatrix}
    |u_t|_x-1 & |v_t|_x-1 & |u_tv_t|_x-1\\
    |u_t|-1 & |v_t|-1 & |u_tv_t|-1\\
    |u_t|_y-1 & |v_t|_y-1 & |u_tv_t|_y-1
\end{bmatrix},
\end{align*}
then we set
\begin{align*}
    W_{t'}=\begin{bmatrix}
    |u_{t'}|-1 & |v_{t'}|-1 & |u_{t'}v_{t'}|-1\\
    |u_{t'}|_x-1 & |v_{t'}|_x-1 & |u_{t'}v_{t'}|_x-1\\
    |u_{t'}|_y-1 & |v_{t'}|_y-1 & |u_{t'}v_{t'}|_y-1
\end{bmatrix}.
\end{align*}
When $\begin{xy}(0,0)*+{t}="A",(10,0)*+{t'}="B",\ar@{-}^k"A";"B" \end{xy}$, like as $\begin{xy}(0,0)*+{t}="A",(10,0)*+{t'}="B",\ar@{-}^1"A";"B" \end{xy}$, if $k$th row of $W_t$ is
\begin{align*}
\begin{bmatrix}|u_{t}|_x-1&|v_{t}|_x-1&|u_{t}v_{t}|_x-1\end{bmatrix}
\end{align*} and
\begin{align*}
\begin{bmatrix}|u_{t}|-1&|v_{t}|-1&|u_{t}v_{t}|-1\end{bmatrix} 
\end{align*} is $\ell$th row of $W_t$, then the $k$th row of $W_{t'}$ is \begin{align*}
\begin{bmatrix}|u_{t'}|-1&|v_{t'}|-1&|u_{t'}v_{t'}|-1\end{bmatrix}
\end{align*}and the $\ell$th row of $W_{t'}$ is 
\begin{align*}
\begin{bmatrix}|u_{t'}|_x-1&|v_{t'}|_x-1&|u_{t'}v_{t'}|_x-1\end{bmatrix}.
\end{align*} The rest row of $W_{t'}$ is 
\begin{align*}
\begin{bmatrix}|u_{t'}|_y-1&|v_{t'}|_y-1&|u_{t'}v_{t'}|_y-1\end{bmatrix}.
\end{align*}
By Theorem \ref{thm:main2}, Lemma \ref{lem:description-dual} and Theorem \ref{dualChristoffel-main} (3), we have $W_t=D(L_t,L)$. Therefore, $W_t^T=D(L,L_t)$ by Proposition \ref{prop:duality}. Therefore, we have $\mathrm{Tree}(V)=\mathrm{Tree}(D)$.
\end{proof}
\bibliography{myrefs}
\end{document}